\documentclass[11pt]{article}
\usepackage{epsfig}
\usepackage{amssymb,amsmath,amsthm,amscd}
\usepackage{latexsym}

\newtheorem{thm}{Theorem}[section]
\newtheorem{prop}[thm]{Proposition}
\newtheorem{cor}[thm]{Corollary}
\newtheorem{lem}[thm]{Lemma}

\theoremstyle{definition}
\newtheorem{defn}[thm]{Definition}

\theoremstyle{remark}

\newcounter{labelflag} 

\newcommand{\Label}[1]{
                       \ifnum\thelabelflag=1
                          \ifmmode
                             \makebox[0in][l]{\qquad\fbox{\rm#1}}
                          \else
                             \marginpar{\vspace{0.7\baselineskip}
                                        \hspace{-1.1\textwidth}
                                        \fbox{\rm#1}}
                          \fi
                       \fi
                       \label{#1} }

\newcommand{\be}{\begin{equation}}
\newcommand{\ee}{\end{equation}}

\newcommand{\fhn}{ FitzHugh-Nagumo equations  }

 \newcommand{\R}{\mathbb{R}}
  \newcommand{\N}{\mathbb{N}}

 \def  \ltwo {L^2 (\R^n)}

\begin{document}

\begin{titlepage}
\title{\Large\bf  Pullback Attractors for the  Non-autonomous
FitzHugh-Nagumo System  on  Unbounded  Domains}
\vspace{7mm}

\author{
Bixiang Wang  \thanks {Supported in part by NSF  grant DMS-0703521.}
\vspace{5mm}\\
Department of Mathematics, New Mexico Institute of Mining and
Technology \vspace{1mm}\\ Socorro,  NM~87801, USA \vspace{5mm}\\
Email: bwang@nmt.edu \qquad Fax: (1-505) 835 5366}
\date{}
\end{titlepage}

\maketitle

\medskip

\begin{abstract}
The existence of a pullback attractor is established for 
the  singularly  perturbed  FitzHugh-Nagumo  system 
defined on the entire space $\R^n$ when external terms
are unbounded in  a phase space. The pullback asymptotic
compactness of  the system  is proved   
 by using uniform a priori estimates
for far-field values of solutions.
Although the limiting system has no global attractor, we
show that the pullback attractors  for the perturbed
system with bounded external terms  are uniformly bounded, and hence do not blow up
as a small parameter approaches zero.
\end{abstract}



\baselineskip=1.4\baselineskip

\section{Introduction}
\setcounter{equation}{0}

In this paper,  we study the dynamical behavior of the
 non-autonomous  FitzHugh-Nagumo equations defined on $\R^n$:
\be
  \label{intr1}
 \frac {\partial u}{\partial t}
 -  \nu \Delta u + \lambda  u + h(u)   + v   =f(t),
 \ee
 \be
 \label{intr2}
 \frac{\partial v}{\partial t}  -   \epsilon (u- \gamma v) =  \epsilon g(t),
\ee
 where $\nu$, $\lambda$,  $\epsilon$ and $\gamma$ are positive constants,
$f$ and $g$ are  given functions depending on $t$, $h$ is a  
nonlinear
function  satisfying a dissipative condition.

 The \fhn   describe  the signal  transmission across
axons  in
neurobiology,   see e.g.,  \cite{bel, fit, nag} and the references therein.
The long time behavior of the   autonomous  \fhn
   was studied  by several  authors in   \cite{liu, luy, mar, mar2, sha}
and the references therein.
  We here   intend to  investigate the dynamical behavior of the
   {\it non-autonomous }  FitzHugh-Nagumo system.

 Global attractors for non-autonomous  dynamical systems
 have been extensively studied in the literature, see, e.g.,  \cite{ant, 
aul1, car1, car2, car3, car4,
cheb1, cheb2, cheb3, che, har,
lan1, lan2,  lu,  moi, pri2, 
sun1, sun2}. 
Particularly,  when  PDEs
 are defined in bounded domains, such attractors
 have been investigated in 
 \cite{car1, car4, cheb1, cheb2, che, har, lu, sun1, sun2}.
 In the case of unbounded domains,  global attractors
 for non-autonomous  PDEs
 with almost periodic external terms 
 have been examined in \cite{ant,     moi, pri2}.
 Notice that  almost periodic  external  terms   are bounded
in a phase space with respect to time.
 It  seems that
  attractors for {\it non-autonomous 
PDEs defined on unbounded domains with 
  unbounded external terms  }
 are not well understood. As far as we know,  
 in this case, the existence of attractors  was 
 established  only for the  Navier-Stokes equation
   by the authors in  \cite{car2, car3} recently.
 In this paper, we  will prove the existence 
 of attractors for the non-autonomous FitzHugh-Nagumo system
  defined on the entire  space  $\R^n$ with unbounded external terms.

Notice that the domain $\R^n$ for  system   \eqref{intr1}-\eqref{intr2}  is    unbounded, 
and the  unboundedness of $\R^n$ introduces a major  obstacle 
 for examining  the    asymptotic compactness of solutions,   since  Sobolev
 embeddings are not compact in this case.
 The difficulty  caused by non-compactness of embeddings can be overcome
by  the  energy equation approach,  which was  introduced  by Ball in
\cite{bal1, bal2}   and then
 used  by several  authors  for autonomous equations in  
  \cite{ gou1, ju1, moi2,  ros1, wanx}
and for non-autonomous equations in \cite{car2, car3, moi}.
   In this paper, we provide
  uniform estimates on the far field values of solutions to
 circumvent the difficulty caused by the unboundedness of the domain.
   This idea was developed in \cite{wan}
 to prove asymptotic compactness of  solutions  for autonomous 
parabolic  equations on $\R^n$,
and   later  extended to non-autonomous  equations  with almost
periodic external terms in \cite{ant, pri2}.
  The   contribution of this paper is to extend the method of using tail estimates
to the case of   non-autonomous   PDEs defined on unbounded domains
with unbounded external terms.

  We first prove that   system   \eqref{intr1}-\eqref{intr2} on $\R^n$
  has a pullback attractor when the parameter $\epsilon$ is a  small but positive number.
  Note that  the limiting  system  with $\epsilon =0$  
 has no global  attractor  since  $v$ is conserved in this case.
  Based on this fact, one may guess that
 the attractors  of the perturbed  system 
 blow up as $\epsilon \to 0$.
 In this respect, we will  demonstrate  that the limiting behavior of 
 the pullback attractors 
 heavily depends on the behavior of the external terms $f$ and $g$.
 If $f$ or $g$ is unbounded  in  a phase space, then it is very likely
 that  the attractors  blow up as $\epsilon \to 0$.
 However, if both $f$ and $g$ are bounded,     the attractors
 are uniformly bounded in a phase with respect to all small but positive
 $\epsilon$. In other words, in this case,  the pullback attractors do 
  not blow up as $\epsilon \to  0$.

 The paper is organized  as follows. In the next section, we
 recall   fundamental concepts and  results 
for  pullback
   attractors  for non-autonomous  dynamical systems. 
   In Section 3,  we derive uniform estimates of solutions
   for the FitzHugh-Nagumo system for large space and time
   variables. Section 4 is devoted to the proof of  existence of a pullback
   attractor for the system. In the last section, we discuss the limiting
   behavior of pullback attractors when $\epsilon \to 0$. Particularly, we
   will show that  all attractors for the perturbed  system are uniformly
   bounded in $H^1(\R^n) \times H^1(\R^n)$ with respect to $\epsilon$
   when external terms are bounded. 
   
   The following notations will be used throughout the paper.
 We denote by
$\| \cdot \|$ and $(\cdot, \cdot)$ the norm and the inner product
in  $L^2(\R^n)$ and use $\| \cdot\|_{p}$    to denote   the norm  in
$L^{p}(\R^n)$.    Otherwise, the
norm of  a general  Banach space $X$  is written as    $\|\cdot\|_{X}$.
 The letters $C$ and $C_i$ ($i=1, 2, \ldots$)
are  generic positive constants  which may change their  values from line to
line or even in the same line.

\section{Preliminaries}
\setcounter{equation}{0}

In this section,  we recall some basic concepts
related to pullback attractors    for non-autonomous  dynamical
systems.  It is  worth to notice that these concepts
are quite similar to that of random attractors for stochastic
systems.  We refer the reader to \cite{arn1, bat1, car22,  car2, car3,  cheb1, chu,   fla1, sun1}
for more details.

Let $\Omega$  be   a nonempty set and  $X$   a metric space with  distance
$d(\cdot, \cdot)$.

\begin{defn}
A   family of mappings $\{\theta_t\}_{t\in \R}$
from $\Omega$ to itself is called a  family of shift operators on $\Omega$ if 
 $\{\theta_t\}_{t\in \R}$ satisfies the  group properties:
 
   (i) \  $\theta_0 \omega =\omega,  \quad  \forall \ \omega \in \Omega;$ 
   
   (ii)\ $  \theta_t (\theta_\tau \omega) = \theta_{t+\tau}  \omega,  \quad
  \forall \ \omega \in \Omega  \quad \mbox{and} \ \  t, \ \tau \in \R.$ 
  \end{defn}

\begin{defn}
Let $\{\theta_t\}_{t\in \R}$
be a   family of shift operators on $\Omega$.  Then a  continuous $\theta$-cocycle
$\phi$ on $X$  
is  a mapping
$$
\phi: \R^+ \times \Omega \times X \to X, \quad (t, \omega, x) \mapsto \phi(t, \omega, x),
$$
which  satisfies, for  all  $\omega \in \Omega$ and
$t,  \tau \in \R^+$,

(i) \  $\phi(0, \omega, \cdot) $ is the identity on $X$;

(ii) \  $\phi(t+\tau, \omega, \cdot) = \phi(t, \theta_\tau \omega, \cdot) \circ \phi(\tau, \omega, \cdot)$;

(iii) \  $\phi(t, \omega, \cdot): X \to  X$ is continuous.
\end{defn}

Hereafter, we always assume that 
$\phi$ is a continuous $\theta$-cocycle on $X$, and   $\mathcal{D}$ a  collection  of families of subsets of $X$:
$$
{\mathcal{D}} = \{ D =\{D(\omega)\}_{\omega \in \Omega}: \ D(\omega) \subseteq X 
\  \mbox{for every} \ \omega \in \Omega \}.
$$

\begin{defn}
Let $\mathcal{D}$ be a collection of families of  subsets of $X$.
Then  $\mathcal{D}$ is called inclusion-closed if 
   $D=\{D(\omega)\}_{\omega \in \Omega} \in {\mathcal{D}}$
and  $\tilde{D}=\{\tilde{D}(\omega) \subseteq X:  \omega \in \Omega\} $
with
  $\tilde{D}(\omega) \subseteq D(\omega)$ for all $\omega \in \Omega$ imply
  that  $\tilde{D} \in {\mathcal{D}}$.
  \end{defn}

\begin{defn}
Let $\mathcal{D}$ be a collection of families of  subsets of $X$ and
$\{K(\omega)\}_{\omega \in \Omega} \in \mathcal{D}$. Then
$\{K(\omega)\}_{\omega \in \Omega} $ is called a  pullback
 absorbing
set for   $\phi$ in $\mathcal{D}$ if for every $B \in \mathcal{D}$
and   $\omega \in \Omega$, there exists $t(\omega, B)>0$ such
that
$$
\phi(t, \theta_{-t} \omega, B(\theta_{-t} \omega)) \subseteq K(\omega)
\quad \mbox{for all} \ t \ge t(\omega, B).
$$
\end{defn}

\begin{defn}
 Let $\mathcal{D}$ be a collection of families of  subsets of $X$.
 Then
$\phi$ is said to be  $\mathcal{D}$-pullback asymptotically
compact in $X$ if  for  every  $\omega \in \Omega$,
$\{\phi(t_n, \theta_{-t_n} \omega,
x_n)\}_{n=1}^\infty$ has a convergent  subsequence  in $X$
whenever
  $t_n \to \infty$, and $ x_n\in   B(\theta_{-t_n}\omega)$   with
$\{B(\omega)\}_{\omega \in \Omega} \in \mathcal{D}$.
\end{defn}

\begin{defn}
 Let $\mathcal{D}$ be a collection of families of  subsets of $X$
 and
 $\{\mathcal{A}(\omega)\}_{\omega \in \Omega} \in {\mathcal{D}}$.
Then     $\{\mathcal{A}(\omega)\}_{\omega \in \Omega}$  
is called a    $\mathcal{D}$-pullback global  attractor  for
  $\phi$
if the following  conditions are satisfied,  for every  $\omega \in \Omega$,

(i) \  $\mathcal{A}(\omega)$ is compact;

(ii) \ $\{\mathcal{A}(\omega)\}_{\omega \in \Omega}$ is invariant, that is,
$$ \phi(t, \omega, \mathcal{A}(\omega)  )
= \mathcal{A}(\theta_t \omega), \ \  \forall \   t \ge 0;
$$

(iii) \ \ $\{\mathcal{A}(\omega)\}_{\omega \in \Omega}$
attracts  every  set  in $\mathcal{D}$,  that is, for every
 $B = \{B(\omega)\}_{\omega \in \Omega} \in \mathcal{D}$,
$$ \lim_{t \to  \infty} d (\phi(t, \theta_{-t}\omega, B(\theta_{-t}\omega)), \mathcal{A}(\omega))=0,
$$
where $d$ is the Hausdorff semi-metric given by
$d(Y,Z) =
  \sup_{y \in Y }
\inf_{z\in  Z}  \| y-z\|_{X}
 $ for any $Y\subseteq X$ and $Z \subseteq X$.
\end{defn}

The following existence result  of  a    pullback global  attractor
for a  continuous cocycle
can be found in \cite{arn1, bat1, car22,  car2, car3,  cheb1, chu,   fla1}.

\begin{prop}
\label{att} Let $\mathcal{D}$ be an  inclusion-closed  collection of families of   subsets of
$X$ and $\phi$ a continuous $\theta$-cocycle on $X$.
 Suppose  that $\{K(\omega)\}_{\omega
\in \Omega} \in {\mathcal{D}} $ is a   closed   absorbing set  for  $\phi$  in
$\mathcal{D}$ and $\phi$ is $\mathcal{D}$-pullback asymptotically
compact in $X$. Then $\phi$ has a unique $\mathcal{D}$-pullback global
attractor $\{\mathcal{A}(\omega)\}_{\omega \in \Omega} \in {\mathcal{D}}$ which is
given by
$$\mathcal{A}(\omega) =  \bigcap_{\tau \ge 0} \  \overline{ \bigcup_{t \ge \tau} \phi(t, \theta_{-t} \omega, K(\theta_{-t} \omega)) }.
$$
\end{prop}

\section{Cocycle associated  with the  FitzHugh-Nagumo system}
\setcounter{equation}{0}

In this  section,  we   construct a $\theta$-cocycle $\phi$
for the  
 non-autonomous  FitzHugh-Nagumo system  defined on $\R^n$:
  for every $\tau \in\R$ and $t > \tau$,
\be
  \label{41}
 \frac {\partial u}{\partial t}
 -  \nu \Delta u + \lambda  u + h(u)   + v   =f(t),
 \ee
 \be
 \label{42}
 \frac{\partial v}{\partial t}  -   \epsilon (u- \gamma v) =  \epsilon g(t),
\ee
 with the initial data
 \be\label{43}
 u(x, \tau) = u_\tau (x), \quad v(x, \tau) = v_\tau (x),  \quad x\in \R^n,
 \ee
where $\nu$, $\lambda$,  $\epsilon$ and $\gamma$ are positive constants,
$f \in L^2_{loc}(\R, \ltwo)$, $g \in L^2(\R, H^1(\R^n) )$, 
 and  $h$ is a smooth
nonlinear
function  that satisfies, for some positive constant $C$,
\be
\label{44}
h(s)s \ge  0, \ \  h(0)=0, \ \ h^{\prime} (s) \ge -C,
 \ \  s \in \R,
\ee
 and
\be
\label{45}
| h^{\prime} (s) | \le C(1+ |s|^r), \ \ s \in \R,
\ee
with $r \ge 0$ for $n \le 2$   and $r\le \min({\frac 4n},
{\frac 2{n-2}})$  for $ n \ge 3$.

By a  standard method,  it can be proved  that 
 if $f \in L^2_{loc} (\R, \ltwo )$, $g \in L^2_{loc} (\R, H^1(\R^n))$
 and 
\eqref{44}-\eqref{45} 
 hold true,  then problem \eqref{41}-\eqref{43}   is well-posed
in  $\ltwo \times  \ltwo$, that  is, for   every  $\tau \in \R$
 and   $ (u_\tau, v_\tau) \in  \ltwo \times  \ltwo
$,  there exists a unique  solution $ (u, v) \in C( [\tau, \infty),  \ltwo
\times  \ltwo) $. Further, the solution is continuous with respect
to initial data  $ (u_\tau, v_\tau)$ in  $ \ltwo \times  \ltwo$.
To construct a cocycle  $\phi$ for problem  \eqref{41}-\eqref{43}, we 
denote by
 $\Omega =\R$,  and   define a shift operator $\theta_t$ 
on $\Omega$
 for every 
$t \in \R$
 by
$$
\theta_t  (\tau ) = t+ \tau, \quad \mbox{for all} \ \ \tau \in \R.
$$
Let $\phi$ be a mapping  from $\R^+ \times \Omega \times  ( \ltwo \times \ltwo)$
to $\ltwo \times \ltwo$ given by
$$
\phi(t, \tau, (u_\tau, v_\tau) ) = 
(u(t+\tau, \tau, u_\tau), v(t+\tau, \tau, v_\tau) ),
$$
 where  $ t \ge 0$,  $  \tau \in \R $, 
$(u_\tau , v_\tau) \in \ltwo \times \ltwo$,  and 
$(u,v)$ is the solution of problem \eqref{41}-\eqref{43}.
By the uniqueness of solutions, we find that
for every $t, s \ge 0$, $\tau \in \R$ and 
$ (u_\tau, v_\tau)  \in \ltwo \times \ltwo$,
$$
\phi (t+s, \tau,  (u_\tau, v_\tau) ) =
\phi (t, s+ \tau,  (\phi(s, \tau, (u_\tau, v_\tau) ) ) ).
$$
Then we see that $\phi$ is a continuous 
$\theta$-cocycle  on $\ltwo \times \ltwo$.
In the next two sections, we will investigate
the existence of a pullback attractor
for $\phi$.  To this end,  we need to define
an appropriate collection of families of subsets
of $\ltwo \times \ltwo$.

 For convenience, 
 if $E \subseteq \ltwo \times \ltwo$,  we denote
   by
   $$ \| E \| = \sup\limits_{x \in E}
   \| x\|_{\ltwo \times \ltwo}.
   $$
 Let
   $D =\{ D(t) \}_{t\in \R}$  be  a family of
   subsets of $\ltwo \times \ltwo$, i.e.,
   $D(t) \subseteq \ltwo \times \ltwo$
   for every $t \in \R$.  In this paper, we are  
   interested in a  family
 $D =\{ D(t) \}_{t\in \R}$ satisfying
 \be
 \label{basin_cond}
 \lim_{t \to  - \infty} e^{    \sigma  t} \| D( t) \|^2 =0,
 \ee
 where $\sigma$ is a positive number given by
 \be
 \label{sigma}
 \sigma = {\frac 12} \epsilon \gamma.
 \ee
 We write the   collection of all
 families satisfying \eqref{basin_cond}
 as  ${\mathcal{D}}_\sigma$, that is,
 \be
 \label{D_sigma}
 {{\mathcal{D}}_\sigma = \{  D =\{ D(t) \}_{t\in \R}:
 D  \ \mbox{satisfies} \  \eqref{basin_cond} \} }.
\ee
Since $\epsilon$ is small in practice, 
we  assume
throughout this paper  that
\be\label{epsilon0}
\epsilon \le \epsilon_0  \quad \mbox{where} \quad   \epsilon_0 =\min\{1, {\frac \lambda\gamma}\}.
\ee
 As we will see later, when we derive uniform estimates of solutions, we need the
 following conditions for the external terms: 
 \be
 \label{fcond}
 \int_{-\infty}^\tau e^{\sigma \xi} \| f(\xi)\|^2 d \xi
<  \infty, \quad \forall \ \tau \in \R,
 \ee
 and
  \be
 \label{gcond}
 \int_{-\infty}^\tau e^{\sigma \xi} \| g(\xi)\|^2_{H^1} d \xi
 <  \infty, \quad \forall \  \tau  \in \R.
 \ee
 In addition,   the following  asymptotically null conditions
are required for proving the asymptotic compactness of solutions: 
 \be
 \label{finfinity}
 \lim_{k  \to \infty} \int_{-\infty}^\tau  \int_{|x| \ge k}  e^{\sigma \xi}   |f(x, \xi) |^2 dx d\xi =0 ,
 \quad \forall \  \tau  \in \R,
 \ee
 and
 \be
 \label{ginfinity}
\lim_{k \to \infty}  \int_{-\infty}^\tau  \int_{|x| \ge k}  e^{\sigma \xi}   |g(x, \xi) |^2 dx d\xi =0,
\quad \forall \  \tau  \in \R.
 \ee
 
 Notice that conditions \eqref{fcond}-\eqref{ginfinity}  
do not require that  $f$ and $g$ be bounded in $L^2(\R^n)$ 
when $t \to \pm \infty$.  Particularly,
These assumptions have no any restriction
 on $f$ and $g$ when $t \to +\infty$.   As a typical example,
 for any   $f_1 \in L^2(\R^n)$ and $g_1 \in H^1(\R^n)$, the functions 
$f(x,t)  =e^{{\frac 14} \sigma |t|}
 f_1 (x)$  and $g(x,t)  =e^{{\frac 14} \sigma |t|}
 g_1 (x)$  satisfy all conditions
 \eqref{fcond}-\eqref{ginfinity}.
 In this case, $f$ and $g$ are indeed unbounded 
 in $L^2(\R^n)$ as $t \to \pm\infty$.
 
 It is  useful to note that  conditions \eqref{finfinity}-\eqref{ginfinity}  
 imply 
 for every $\tau \in \R$ and $\eta>0$, there is $K=K(\tau, \eta)>0$
 such that
\be
\label{finfinity2}
 \int_{-\infty}^\tau  \int_{|x| \ge K}  e^{\sigma \xi}   |f(x, \xi) |^2 dx d\xi \le \eta e^{\sigma \tau} ,
\ee
 and
\be
\label{ginfinity2}
 \int_{-\infty}^\tau  \int_{|x| \ge K}  e^{\sigma \xi}   |g(x, \xi) |^2 dx d\xi \le \eta e^{\sigma \tau} .
\ee
We remark that \eqref{finfinity2} and \eqref{ginfinity2} will play
a crucial role when we derive uniform estimates on the tails of solutions
in the next  section.

\section{Uniform     estimates of solutions }
\setcounter{equation}{0}

      In this section, we
 derive uniform estimates of  solutions  of  problem \eqref{41}-\eqref{43}  defined on $\R^n$
when $t \to \infty$. These estimates are necessary  for   proving  the existence of
a bounded pullback  absorbing set
 and the  pullback asymptotic compactness of the $\theta$-cocycle $\phi$
   associated  with the system.
  In particular, we
  will  show that   the tails of the solutions, i.e., solutions evaluated at large values of $|x|$, are uniformly small
 when time is sufficiently  large.
 
 We  start with the estimates in $\ltwo \times \ltwo$.

\begin{lem}
\label{lem42}
 Suppose  \eqref{44}-\eqref{45}  and \eqref{fcond}-\eqref{gcond} hold.
Then for every $\tau \in \R$ and $D=\{D(t)\}_{t\in \R} \in {\mathcal{D}}_\sigma$,
 there exists  $T=T(\tau, D)>0$ such that for all $t \ge T$,
$$
\| u(\tau, \tau -t, u_0(\tau -t) ) \|^2 + \| v(\tau, \tau -t, v_0(\tau -t) ) \| ^2 \le M  e^{- \sigma  \tau}
\int_{-\infty}^\tau
e^{\sigma \xi}  \left ( \| f(\xi )\|^2
 +   \| g(\xi )\|^2 \right )  d \xi ,
$$
and
 $$
\int_{\tau -t}^{\tau}  e^{\sigma  \xi } \|   u (\xi, \tau -t, u_0(\tau -t) )  \|^2_{H^1} d\xi
  \le M 
    \int_{-\infty}^\tau
e^{\sigma \xi}  \left ( \| f(\xi )\|^2
 +   \| g(\xi )\|^2 \right )  d \xi
 ,
 $$
where $M$ is a  positive constant depending   on the data     $(
\nu,  \lambda,  \epsilon, \gamma)$.
\end{lem}

 \begin{proof}
Taking the inner product of  \eqref{41}
 with $ \epsilon u$ in $\ltwo$, we find that
\be
\label{p42_1}
{\frac 12} \epsilon {\frac d{dt}} \| u \|^2
 +  \epsilon   \nu \| \nabla u \|^2
 +  \epsilon \lambda \| u \|^2 + \epsilon  \int  h(u)u
 + \epsilon \int u v  = \epsilon \ \int  f(t) u.
\ee
Taking the inner product of  \eqref{42}
 with $  v$ in $\ltwo$, we find
\be
\label{p42_2}
{\frac 12}   {\frac {d}{dt}} \| v \|^2
 + \epsilon  \gamma  \| v \|^2 - \epsilon \int  uv
 = \epsilon \int  g(t) v.
\ee
It follows from \eqref{p42_1}-\eqref{p42_2} that
\be
\label{p42_3}
{\frac 12} {\frac d{dt}} \left ( \epsilon  \| u \|^2
  +    \| v \|^2 \right )
   +  \epsilon  \nu \| \nabla u \|^2
 +  \epsilon  \lambda \| u \|^2
 +   \epsilon \gamma  \| v \|^2
 + \epsilon    \int  h(u)u
 =   \epsilon   \int  f(t)  u +   \epsilon \int  g(t) v.
 \ee
  Note that the terms on the  right-hand side of  \eqref{p42_3}
   are  bounded  by
\be
\label{p42_4}
|\epsilon \int f(t) u | \le  \epsilon  \| f (t) \| \| u \|
  \le {\frac 12}  \epsilon \lambda \| u \|^2
  + {\frac {\epsilon}{2 \lambda}} \| f \|^2 ,
 \ee
and
\be
\label{p42_5}
|\epsilon \int g(t) v | \le  \epsilon   \| g(t) \| \| v \|
  \le {\frac 12} \epsilon \gamma  \| v \|^2
  + {\frac {\epsilon}{2 \gamma}} \| g \|^2 .
 \ee
 By  \eqref{p42_3}-\eqref{p42_5} and \eqref{44},  we  obtain
$$
{\frac d{dt}} \left ( \epsilon \| u \|^2
  +   \| v \|^2 \right )
   +  2\epsilon  \nu \| \nabla u \|^2
 +  \epsilon \lambda \| u \|^2
 + \epsilon \gamma  \| v \|^2
  \le {\frac {\epsilon}{\lambda}} \| f \|^2
  + {\frac {\epsilon}{\gamma}} \| g \|^2 ,
$$
and hence by \eqref{sigma} and \eqref{epsilon0} we have
\be
\label{p42_8}
{\frac d{dt}} \left ( \epsilon \| u \|^2
  +   \| v \|^2 \right )
   +  2 \sigma \left (   \epsilon \| u \|^2  +   \| v \|^2
   \right )
   +  2\epsilon  \nu \| \nabla u \|^2
  \le {\frac {\epsilon}{\lambda}} \| f \|^2
  + {\frac {\epsilon}{\gamma}} \| g \|^2 .
  \ee
 Multiplying  \eqref{p42_8} by $e^{\sigma t}$ and
then  integrating  between $\tau -t $ and $\tau$ with $t \ge 0$,  we get,
 $$
\epsilon  \|  u(\tau, \tau -t, u_0(\tau-t)  ) \|^2 +   \| v(\tau, \tau -t, v_0(\tau -t)) \|^2
$$
$$
+  \sigma  
\int_{\tau -t}^\tau e^{\sigma (\xi -\tau)}
 \left (
\|   u(\xi, \tau -t, u_0(\tau -t) )\|^2 +  \|   v (\xi, \tau -t, v_0(\tau -t) )\|^2
\right )
 d \xi
$$
$$
+ 2 \epsilon \nu
\int_{\tau -t}^\tau e^{\sigma (\xi -\tau)}
\| \nabla u(\xi, \tau -t, u_0(\tau -t) )\|^2 d \xi
$$
$$
\le e^{- \sigma \tau }e^{ \sigma (\tau-t) } \left ( \epsilon \| u_0(\tau -t ) \|^2
  +   \| v_0(\tau -t) \|^2 \right )
  $$
  $$
  + {\frac {\epsilon}{   \lambda}} e^{- \sigma \tau} \int_{\tau -t}^\tau
e^{\sigma \xi} \| f(\xi )\|^2 d \xi
 + {\frac {\epsilon}{   \gamma }} e^{- \sigma \tau} \int_{\tau -t}^\tau
e^{\sigma \xi} \| g(\xi )\|^2 d \xi
$$
$$
\le e^{- \sigma \tau }e^{ \sigma (\tau-t) } \left ( \epsilon \| u_0(\tau -t ) \|^2
  +   \| v_0(\tau -t) \|^2 \right )
  $$
   \be
 \label{p42_9}
  + {\frac {\epsilon}{   \lambda}} e^{- \sigma \tau} \int_{-\infty}^\tau
e^{\sigma \xi} \| f(\xi )\|^2 d \xi
 + {\frac {\epsilon}{   \gamma }} e^{- \sigma \tau} \int_{-\infty}^\tau
e^{\sigma \xi} \| g(\xi )\|^2 d \xi   .
\ee
Notice that $(u_0(\tau -t), v_0(\tau -t) ) \in D(\tau -t)$
and $D= \{D(t)\}_{t \in \R} \in {\mathcal{D}}_\sigma$. We find that
for every $\tau \in \R$, there exists $T=T(\tau, D)$ such that
for all $t \ge T$,
$$
e^{ \sigma (\tau-t) } \left ( \epsilon \| u_0(\tau -t ) \|^2
  +   \| v_0(\tau -t) \|^2 \right )
  \le
   {\frac {\epsilon}{   \lambda}}   \int_{-\infty}^\tau
e^{\sigma \xi} \| f(\xi )\|^2 d \xi
 + {\frac {\epsilon}{   \gamma }}   \int_{-\infty}^\tau
e^{\sigma \xi} \| g(\xi )\|^2 d \xi  ,
$$
which along with  \eqref{p42_9}
shows that, for all $t \ge T$,
$$
\epsilon  \|  u(\tau, \tau -t, u_0(\tau-t)  ) \|^2 +   \| v(\tau, \tau -t, v_0(\tau -t)) \|^2
$$
$$
+  \sigma  
\int_{\tau -t}^\tau e^{\sigma (\xi -\tau)}
 \left (
\|   u(\xi, \tau -t, u_0(\tau -t) )\|^2 +  \|   v (\xi, \tau -t, v_0(\tau -t) )\|^2
\right )
 d \xi
$$
$$
+ 2 \epsilon \nu
\int_{\tau -t}^\tau e^{\sigma (\xi -\tau)}
\| \nabla u(\xi, \tau -t, u_0(\tau -t) )\|^2 d \xi
$$
\be
\label{p42_10}
\le e^{-\sigma \tau} \left (
 {\frac {2\epsilon}{   \lambda}}   \int_{-\infty}^\tau
e^{\sigma \xi} \| f(\xi )\|^2 d \xi
 + {\frac {2\epsilon}{   \gamma }}   \int_{-\infty}^\tau
e^{\sigma \xi} \| g(\xi )\|^2 d \xi
\right ),
\ee
which completes the proof.
  \end{proof}

We will need the following estimates when proving the
asymptotic compactness of solutions,  which can be derived
in a similar manner as Lemma \ref{lem42}.

\begin{lem}
\label{lem43}
 Suppose  \eqref{44}-\eqref{45}  and \eqref{fcond}-\eqref{gcond} hold.
Then for every $\tau \in \R$ and $D=\{D(t)\}_{t\in \R} \in {\mathcal{D}}_\sigma$,
 there exists  $T=T(\tau, D)>1$ such that for all $t \ge T$,
 $$
 \int_{\tau -1}^\tau e^{\sigma \xi} \left (   \| u(\xi, \tau -t, u_0(\tau -t) ) \|^2
 + \| v(\xi, \tau -t, v_0(\tau -t) ) \|^2 \right ) d\xi
  \le M 
    \int_{-\infty}^\tau
e^{\sigma \xi}  \left ( \| f(\xi )\|^2
 +   \| g(\xi )\|^2 \right )  d \xi
 ,
 $$
 $$
\int_{\tau -1}^{\tau}  e^{\sigma \xi } \|  \nabla  u (\xi, \tau -t, u_0(\tau -t) )  \|^2  d\xi
  \le M 
    \int_{-\infty}^\tau
e^{\sigma \xi}  \left ( \| f(\xi )\|^2
 +   \| g(\xi )\|^2 \right )  d \xi
 ,
 $$
where $M$ is a  positive constant depending   on the data     $(
\nu,  \lambda,  \epsilon, \gamma)$.
\end{lem}

\begin{proof}
 Note that \eqref{p42_8} implies that
 \be
 \label{p43_1}
 {\frac d{dt}} \left ( \epsilon \| u \|^2
  +   \| v \|^2 \right )
   +    \sigma \left (   \epsilon \| u \|^2  +   \| v \|^2
   \right )
  \le {\frac {\epsilon}{\lambda}} \| f \|^2
  + {\frac {\epsilon}{\gamma}} \| g \|^2.
\ee
Multiplying \eqref{p43_1} by $e^{\sigma t}$ and integrating
over $(\tau -1, \tau -t)$ with $t \ge 1$, by repeating the proof of
\eqref{p42_10} we find that there exists $T=T(\tau, D)>1$ such that
for all $t \ge T$,
\be
\label{p43_2}
 \| u(\tau -1, \tau -t, u_0(\tau -t) ) \|^2 + \| v(\tau, \tau -t, v_0(\tau -t) ) \| ^2 \le M  e^{- \sigma  \tau}
\int_{-\infty}^\tau
e^{\sigma \xi}  \left ( \| f(\xi )\|^2
 +   \| g(\xi )\|^2 \right )  d \xi.
 \ee
Multiplying \eqref{p42_8} by $e^{\sigma t}$ and  then integrating
over $(\tau -1, \tau )$,  by \eqref{p43_2}we get that, for all $t\ge T$,
$$
e^{\sigma \tau}  \left (\epsilon \| u(\tau, \tau -t, u_0(\tau -t) ) \|^2
 +
 \| v (\tau , \tau -t, v_0(\tau -t) ) \|^2 \right )
 $$
 $$
 + \sigma \int_{\tau -1}^\tau e^{\sigma \xi} \left (\epsilon  \| u(\xi, \tau -t, u_0(\tau -t) ) \|^2
 + \| v(\xi, \tau -t, v_0(\tau -t) ) \|^2 \right ) d\xi
 $$
 $$
 + 2 \epsilon \nu  \int_{\tau -1}^\tau e^{\sigma \xi} \| \nabla u (\xi, \tau -t, u_0(\tau -t) ) \|^2 d\xi
 $$
 $$ \le 
 e^{\sigma (\tau -1 )}  \left (\epsilon \| u(\tau -1, \tau -t, u_0(\tau -t) ) \|^2
 +
 \| v (\tau -1 , \tau -t, v_0(\tau -t) ) \|^2 \right )
 $$
 $$
 + {\frac {\epsilon}{\lambda}} \int_{\tau -1}^\tau e^{\sigma \xi}  \| f (\xi) \|^2 d\xi
  + {\frac {\epsilon}{\gamma}}  \int_{\tau -1}^\tau e^{\sigma \xi} \| g(\xi)  \|^2  d\xi.
  $$
  $$
  \le c \int_{-\infty}^\tau e^{\sigma \xi}  \left (  \| f (\xi) \|^2 + \| g(\xi)\|^2 \right ) d\xi,
  $$
  which completes the proof.
\end{proof}

 \begin{lem}
\label{lem44}
Suppose  \eqref{44}-\eqref{45}  and \eqref{fcond}-\eqref{gcond} hold.
Then for every $\tau \in \R$ and $D=\{D(t)\}_{t\in \R} \in {\mathcal{D}}_\sigma$,
 there exists  $T=T(\tau, D)>1$ such that for all $t \ge T$,
 $$
    \| \nabla u (\tau, \tau -t, u_0(\tau -t) ) \|^2
\le M e^{-\sigma \tau} \int_{-\infty}^\tau e^{\sigma \xi} \left ( \| f(\xi) \|^2 + \| g(\xi) \|^2 \right ) d\xi,
 ,
 $$
where $M$ is a  positive constant depending   on the data     $(
\nu,  \lambda,  \epsilon, \gamma)$.
\end{lem}

 \begin{proof}
 Taking the inner product
 of  \eqref{41} with $-\Delta u$ in $\ltwo$, we get
\be
\label{p44_1}
{\frac 12} {\frac d{dt}} \| \nabla u \|^2
 + \nu \| \Delta u \|^2 + \lambda \| \nabla u \|^2
 = \int h(u) \Delta u
 +   \int v \Delta u -   \int  f (t)  \Delta u.
 \ee
 We now estimate the   right-hand side
 of  \eqref{p44_1}.  For the last term, we have
\be
\label{p44_2}
| \int f (t)  \Delta u | \le \| f(t) \| \|
 \Delta u \|
  \le {\frac 14} \nu \| \Delta u \|^2 +{\frac 1{ \nu}}
 \| f (t) \|^2.
\ee 
  For the second  term on the right-hand side of
 \eqref{p44_1},  we have  the following bounds
\be
\label{p44_3}
|    \int  v \Delta u | \le     \| v \| \|
 \Delta u \|  
  \le {\frac 14} \nu \| \Delta u \|^2 +  {\frac 1\nu} \| v \|^2 .
 \ee
 Note that by \eqref{44}, the first term on the right-hand
 side of \eqref{p44_1} is bounded by
 \be
 \label{p44_4}
\int h(u) \Delta u = - \int h^{\prime} (u) | \nabla u |^2 \le C
 \| \nabla u \|^2,
\ee
 where $C$ is the constant in \eqref{44}.
Then it follows from \eqref{p44_1}-\eqref{p44_4} that
\be
\label{p44_5}
{\frac d{dt}} \| \nabla u \|^2  + \sigma \| \nabla u \|^2
\le  C   \| \nabla u \|^2  
+ {\frac 2\nu} \| v \|^2  +{\frac 2{ \nu}}
 \| f(t) \|^2 .
\ee
Multiplying \eqref{p44_5} by $e^{\sigma t}$ and then integrating
the resulting equality
over $(s, \tau)$  with $\tau -1 \le s \le \tau$, we find that
$$
e^{\sigma \tau} \| \nabla u (\tau, \tau -t, u_0(\tau -t) ) \|^2
\le e^{\sigma s} \| \nabla u (s, \tau -t, u_0(\tau -t) ) \|^2
$$
$$
+ C \int_s^\tau e^{\sigma \xi} \| \nabla u(\xi, \tau -t, u_0(\tau -t) ) \|^2 d \xi
$$
$$
+ {\frac 2\nu} \int_s^\tau e^{\sigma \xi} \| v(\xi, \tau -t, v_0(\tau -t) ) \|^2 d \xi
+ {\frac 2\nu} \int_s^\tau e^{\sigma \xi} \|  f(\xi)  \|^2 d \xi
$$
$$
\le e^{\sigma s} \| \nabla u (s, \tau -t, u_0(\tau -t) ) \|^2
+ C \int_{\tau -1}^\tau e^{\sigma \xi} \| \nabla u(\xi, \tau -t, u_0(\tau -t) ) \|^2 d \xi
$$
\be
\label{p44_6}
+ {\frac 2\nu} \int_{\tau -1}^\tau e^{\sigma \xi} \| v(\xi, \tau -t, v_0(\tau -t) ) \|^2 d \xi
+ {\frac 2\nu} \int_{-\infty}^\tau e^{\sigma \xi} \|  f(\xi)  \|^2 d \xi.
\ee
We now integrate \eqref{p44_6} with respect to $s$ over $(\tau -1, \tau)$ to get
$$
e^{\sigma \tau} \| \nabla u (\tau, \tau -t, u_0(\tau -t) ) \|^2
$$
$$
\le \int_{\tau -1}^\tau e^{\sigma s} \| \nabla u (s, \tau -t, u_0(\tau -t) ) \|^2 ds
+ C \int_{\tau -1}^\tau e^{\sigma \xi} \| \nabla u(\xi, \tau -t, u_0(\tau -t) ) \|^2 d \xi
$$
\be
\label{p44_7}
+ {\frac 2\nu} \int_{\tau -1}^\tau e^{\sigma \xi} \| v(\xi, \tau -t, v_0(\tau -t) ) \|^2 d \xi
+ {\frac 2\nu} \int_{-\infty}^\tau e^{\sigma \xi} \|  f(\xi)  \|^2 d \xi.
\ee
Then it follows from \eqref{p44_7} and 
   Lemma \ref{lem43} that   there is $T=T(\tau, D)>1$ such that
 for all $t\ge T$,
 $$
e^{\sigma \tau} \| \nabla u (\tau, \tau -t, u_0(\tau -t) ) \|^2
\le C \int_{-\infty}^\tau e^{\sigma \xi} \left ( \| f(\xi) \|^2 + \| g(\xi) \|^2 \right ) d\xi,
$$
which completes the proof.  \end{proof}

Note that  Lemma \ref{lem44} shows that
system \eqref{41}-\eqref{42}  has smoothing effect
on the  $u$ components of   solutions. However  this is not true for
the $v$ components. In order  to establish the  uniform asymptotic compactness
of  $v$, we need to decompose
$v$  as a sum of two functions: one is  regular in the sense it belongs to
$H^1 (\R^n)$ and the other converges to zero   as $t \to \infty$.
This splitting  technique  was used by several  authors for  the autonomous
FitzHugh-Nagumo equations in bounded domains (see, for example~\cite{mar}). We split  $v$ as
$v=v_1 + v_2$ where $v_1$ is the  solution of the initial value problem, for $t \ge s$ with $s \in \R$,
\be
\label{v1}
{\frac {\partial v_1}{\partial t}}  +\epsilon \gamma v_1 =0,\quad
 v_1(s) =  v_0,
\ee
and $v_2$ is the  solution of
\be \label{v2}
{\frac {\partial v_2}{\partial t}} + \epsilon \gamma v_2 - \epsilon u = \epsilon g(t) ,\quad
  v_2(s)  =  0.
\ee
It is evident  that $v_1$ satisfies:
$$
\|v_1 (\tau) \|  = e^{-\epsilon \gamma (\tau -s)} \|v_1(s)\|,
\quad \mbox{for all} \  \tau  \ge s.
$$
Given $\tau \in \R$ and $t\ge0$, set $s= \tau -t$. Then we get that
\be
 \label{v1_ener}
\|v_1 (\tau) \|  = e^{-\epsilon \gamma  \tau   } e^{\epsilon \gamma (\tau -t)} \|v_0(\tau -t )\|,
\ee
which implies  that $v_1$ converges to zero  when $t \to \infty$. 
Next, we derive    uniform estimates for  $v_2$ in $H^1(\R^n)$.

 \begin{lem}
\label{lemv2}
 Suppose  \eqref{44}-\eqref{45}  and \eqref{fcond}-\eqref{gcond} hold.
Then for every $\tau \in \R$ and $D=\{D(t)\}_{t\in \R} \in {\mathcal{D}}_\sigma$,
 there exists  $T=T(\tau, D)>0$ such that for all $t \ge T$,
$$ 
\| \nabla v_2 (\tau, \tau -t, 0) \|^2 \le M  e^{- \sigma  \tau}
\int_{-\infty}^\tau
e^{\sigma \xi}  \left ( \| f(\xi )\|^2
 +   \| g(\xi )\|^2_{H^1} \right )  d \xi ,
$$
where $M$ is a  positive constant depending   on the data     $(
\nu,  \lambda,  \epsilon, \gamma)$.
  \end{lem}

\begin{proof}
Taking the inner product of \eqref{v2} with $- \Delta v_2$
in $\ltwo$,  we  obtain that
\be
\label{pv2_1}
{\frac{d}{dt}} \| \nabla v_2 \| ^2 +  2  \epsilon \gamma \| \nabla v_2 \|^2
=   2 \epsilon  (\nabla v_2, \nabla u ) -  2 \epsilon \int g(t)  \Delta v_2    \ dx.
\ee
The first term  on the right-hand side of \eqref{pv2_1} is bounded by
\be
\label{pv2_1_a1}
| 2 \epsilon  (\nabla v_2, \nabla u ) |
\le 2 \epsilon \| \nabla v_2 \|  \| \nabla u \|
\le {\frac 14}  \epsilon \gamma  \| \nabla v_2 \|^2 +  {\frac {4 \epsilon}\gamma}   \| \nabla u \|^2.
  \ee
  For  the second term on the right-hand side of \eqref{pv2_1} we have
  \be
  \label{pv2_1_a2}
  | 2 \epsilon \int g(t)  \Delta v_2    \ dx |
  \le {\frac 14}  \epsilon \gamma  \| \nabla v_2 \|^2 +  {\frac {4 \epsilon}\gamma}   \| \nabla g \|^2.
  \ee
  Then it follows from  \eqref{pv2_1}-\eqref{pv2_1_a2} that
 \be
\label{pv2_2}
{\frac{d}{dt}} \| \nabla v_2 \| ^2 +  \sigma  \| \nabla v_2 \|^2
\le
{\frac {4 \epsilon}\gamma}   \| \nabla u \|^2
+  {\frac {4 \epsilon}\gamma}   \| \nabla g  \|^2 .
  \ee
  Multiplying \eqref{pv2_2} by $e^{\sigma t}$, and then integrating
  the resulting inequality over $(\tau, \tau -t)$ with $t\ge 0$, we obtain that
  $$
  e^{\sigma \tau} \| \nabla v_2 (\tau, \tau -t, 0) \|^2
  \le  
{\frac {4 \epsilon}\gamma} \int_{\tau -t}^\tau e^{\sigma \xi}   \| \nabla u (\xi, \tau -t, u_0(\tau -t ) ) \|^2 d\xi
 +  {\frac {4 \epsilon}\gamma} \int_{\tau -t}^\tau e^{\sigma \xi}   \| \nabla g (\xi)  \|^2  d\xi,
$$
 which along with Lemma \ref{lem42} shows that  there is  $T=T(\tau, D)>0$ such that for
 all $t \ge T$,
   $$
  e^{\sigma \tau} \| \nabla v_2 (\tau, \tau -t, 0) \|^2
  \le C \int_{-\infty}^\tau e^{\sigma \xi} ( \| f(\xi) \|^2 + \| \nabla g(\xi) \|^2 ) d \xi.
  $$ The proof is completed.
\end{proof}

Next, we establish  uniform estimates on the tails
of solutions when $t \to \infty$. We show that the tails of solutions
are  uniformly  small  for  large space and time variables. 
    These uniform estimates are crucial for proving
the pullback asymptotic compactness of the cocycle $\phi$.

 \begin{lem}
\label{lem45}
Suppose  \eqref{44}-\eqref{45}  and \eqref{fcond}-\eqref{ginfinity} hold.
Then for every $\eta>0$, $\tau \in \R$ and $D=\{D(t)\}_{t\in \R} \in {\mathcal{D}}_\sigma$,
 there exists  $T= T(\tau, D, \eta)>0$ and 
$K=K(\tau, \eta)>0 $ such that for all $t \ge T$ and $k \ge K$,
$$
 \int _{|x| \ge  k} \left (   |u(x, \tau, \tau -t, u_0(\tau -t) )|^2
  +  |v(x, \tau, \tau -t, v_0(\tau -t) )|^2  \right ) dx
\le \eta,
  $$
where  $ (u_0(\tau -t),  v_0(\tau -t) ) \in D(\tau -t)$; 
  $K(\tau, \eta)$  depends   on $\tau$,   $\eta$   and  the data     $(
\nu,  \lambda,  \epsilon, \gamma )$;   $T(\tau, D, \eta)$ depends
on  $\tau$, $D$,  $\eta$ and  the data $( \nu, \lambda, \epsilon,  \gamma )$
\end{lem}

 \begin{proof}
  We  use a cut-off technique  to establish the estimates on the tails of solutions.
   Let $\theta$  be  a smooth function   satisfying   $0 \le \theta
(s) \le 1$ for $ s \in \R^+$, and
$$
\theta (s) =0 \ \mbox{for} \  0 \le s \le 1;
 \ \  \theta (s) =1 \
 \mbox{for} \ s \ge 2.
$$
Then there exists
  a constant
 $C$ such that  $ | \theta^{\prime}  (s) | \le C$
 for $  s \in \R^+ $.
Taking the inner product of \eqref{41} with
 $ \epsilon \theta ({\frac {|x|^2}{k^2}}) u $ in  $\ltwo$, we get
$$
{\frac 12} \epsilon  {\frac
 d{dt}} \int \theta ({\frac {|x|^2}{k^2}}) |u|^2
  -  \epsilon \nu \int \theta ({\frac {|x|^2}{k^2}}) u  \Delta u
  +  \epsilon \lambda \int \theta ({\frac {|x|^2}{k^2}}) |u|^2
$$
\be
\label{p45_1}
=- \epsilon \int \theta ({\frac {|x|^2}{k^2}} )
 h(u) u  - \epsilon   \int \theta ({\frac {|x|^2}{k^2}} )uv
  +  \epsilon \int \theta ({\frac {|x|^2}{k^2}})   u  f(t) .
 \ee
 Taking the inner product of \eqref{42} with
  $   \theta ({\frac {|x|^2}{k^2}}) v$ in $\ltwo$,
  we find
  \be
\label{p45_2}
{\frac 12}    {\frac
 d{dt}} \int \theta ({\frac {|x|^2}{k^2}}) |v|^2
  +  \epsilon \gamma   \int \theta ({\frac {|x|^2}{k^2}})   |v|^2
=  \epsilon   \int \theta ({\frac {|x|^2}{k^2}} )u v
  +   \epsilon \int \theta ({\frac {|x|^2}{k^2}})  v g(t) .
  \ee
  Summing up \eqref{p45_1} and  \eqref{p45_2},
 by \eqref{44} we  obtain that
$$
{\frac 12}   {\frac d{dt}}
  \int \theta ({\frac {|x|^2}{k^2}}) \left ( \epsilon |u|^2
  +  |v|^2  \right )
  +   \epsilon     \int  \theta ({\frac {|x|^2}{k^2}})  \left ( \lambda  |u|^2
  +     \gamma |v|^2 \right )
$$
\be
\label{p45_3}
\le  \epsilon  \nu \int \theta ({\frac {|x|^2}{k^2}}) u  \Delta u
   +  \epsilon  \int  \theta ({\frac {|x|^2}{k^2}}) u  f(t)
   +   \epsilon  \int  \theta ({\frac {|x|^2}{k^2}})  v g(t) .
  \ee
 We now  estimate    the right-hand side of
\eqref{p45_3}. For the
 second term we have
 $$ \epsilon  \int_{\R^n}  \theta (\frac{|x|^2}{k^2})
  u  f (t) = \epsilon \int_{|x| \ge k} \theta (\frac{|x|^2}{k^2})
   u  f (t) $$
  $$
\leq
\frac12  \epsilon \lambda  \int_{|x|\geq
k}\theta^2 (\frac{|x|^2}{k^2} ) \left |u \right|^2 +\frac
\epsilon{2\lambda}\int_{|x|\geq k} |f(x, t)|^2 $$
  \be
  \label{p45_4}
   \leq
 \frac12  \epsilon  \lambda \int_{\R^n} \theta (\frac{|x|^2}{k^2})
 \left|u\right|^2 +\frac \epsilon{2\lambda}
 \int_{|x|\geq k} \left|f (x, t)
 \right|^2.
 \ee
 Similarly, for the last term on the right-hand side of
 \eqref{p45_3}, we find that  
 \be
 \label{p45_8}
 \epsilon \int_{\R^n}   \theta (\frac{|x|^2}{k^2})
  v  g (t)
 \le
 \frac12  \epsilon \gamma  \int \theta (\frac{|x|^2}{k^2})
 \left|v\right|^2
 + 
 {\frac {\epsilon}{2 \gamma}} \int_{|x| \ge k}
 |g(x,t)|^2 dx .
\ee
 On the other hand,
 for the first term on the  right-hand side of \eqref{p45_3},
 by integration by parts, we have
 $$
 \epsilon \nu \int_{\R^n} \theta ({\frac {|x|^2}{k^2}} ) u \Delta u
  = - \epsilon \nu \int _{\R^n}  \theta
 ({\frac {|x|^2}{k^2}})  | \nabla u|^2
 - \epsilon  \nu \int_{\R^n}   \theta^{\prime} ({\frac {|x|^2}{k^2}})
 (  {\frac {2x}{k^2}} \cdot  \nabla u ) u .
 $$
 \be
 \label{p45_9}
\le
  - \epsilon  \nu  \int _{ k \le |x| \le {\sqrt{2}}k}  \theta^{\prime} ({\frac {|x|^2}{k^2}})
 (  {\frac {2x}{k^2}} \cdot  \nabla u ) u
 \le  {\frac {\epsilon M}k}  \int_{k \le |x| \le {\sqrt{2}}k } | u|  | \nabla u |
 \le 
 {\frac {\epsilon M}{2k}} (  \| u \|^2 +  \| \nabla u \|^2 )
,
\ee
where $M $  is independent of   $\epsilon$ and $k$.
By   \eqref{p45_3} and \eqref{p45_4}-\eqref{p45_9},   we find  that
 $$
    {\frac d{dt}}
  \int \theta ({\frac {|x|^2}{k^2}}) \left ( \epsilon |u|^2
  +  |v|^2  \right )
  +  \sigma    \int  \theta ({\frac {|x|^2}{k^2}})  \left ( \epsilon  |u|^2
  +      |v|^2 \right ) 
  $$
   \be
 \label{p45_9_a1}
  \le  {\frac \epsilon{\lambda}} \int_{|x| \ge k} |f(x,t)|^2 dx
  + {\frac \epsilon{\gamma}} \int_{|x| \ge k} |g(x,t)|^2 dx
  + {\frac {\epsilon M}{k}} ( \| u \|^2 + \| \nabla u \|^2 ).
\ee
Multiplying \eqref{p45_9_a1} by $e^{\sigma t}$ and then
integrating over $(\tau -t, \tau)$ with $t\ge 0$, we get that
$$
 \int \theta ({\frac {|x|^2}{k^2}}) \left ( \epsilon |u(x, \tau, \tau -t, u_0(\tau -t) )|^2
  +  |v(x, \tau, \tau -t, v_0(\tau -t) )|^2  \right ) dx
  $$
  $$
  \le e^{-\sigma \tau}e^{\sigma (\tau - t)}
  \int \theta ({\frac {|x|^2}{k^2}})
  \left (
    \epsilon |u_0(x, \tau -t ) |^2 + | v_0 (x, \tau -t ) |^2
  \right ) dx
  $$
  $$
  + {\frac \epsilon{\lambda}} e^{-\sigma \tau} \int_{\tau -t}^\tau \int_{|x| \ge k}
e^{\sigma \xi}   |f(x, \xi) |^2 dx d\xi
+ {\frac \epsilon{\gamma}} e^{-\sigma \tau} \int_{\tau -t}^\tau \int_{|x| \ge k}
e^{\sigma \xi}   |g(x, \xi) |^2 dx d\xi
  $$
  $$
  + {\frac {\epsilon M}{k}} e^{-\sigma \tau} \int_{\tau -t}^\tau \ 
e^{\sigma \xi}  \left (
 \| u (\xi, \tau -t, u_0(\tau -t ) )\|^2 
 + \| \nabla u (\xi, \tau -t, u_0(\tau -t ) )\|^2 
\right )
d \xi
$$
$$
  \le  e^{-\sigma \tau}e^{\sigma (\tau - t)}
  \left (
    \epsilon \|u_0( \tau -t ) \|^2 + \| v_0 (\tau -t ) \|^2
  \right ) 
  $$
  $$
  + {\frac \epsilon{\lambda}} e^{-\sigma \tau} \int_{-\infty}^\tau \int_{|x| \ge k}
e^{\sigma \xi}   |f(x, \xi) |^2 dx d\xi
+ {\frac \epsilon{\gamma}} e^{-\sigma \tau} \int_{-\infty}^\tau \int_{|x| \ge k}
e^{\sigma \xi}   |g(x, \xi) |^2 dx d\xi
  $$
\be
\label{p_45_11}
  + {\frac {\epsilon M}{k}} e^{-\sigma \tau} \int_{\tau -t}^\tau \ 
e^{\sigma \xi}  \left (
 \| u (\xi, \tau -t, u_0(\tau -t ) )\|^2 
 + \| \nabla u (\xi, \tau -t, u_0(\tau -t ) )\|^2 
\right )
d \xi .
\ee
Note that for given $\eta>0$, there is $T_1 =T_1(\tau, D, \eta)>0$
such that for all $t \ge T_1$,
\be
\label{p_45_12}
e^{-\sigma \tau}e^{\sigma (\tau - t)}
  \left (
    \epsilon \|u_0( \tau -t ) \|^2 + \| v_0 (\tau -t ) \|^2
  \right )  \le \eta.
  \ee
  On the other hand, by \eqref{finfinity2}-\eqref{ginfinity2}
  we  find that there is $K_1=K_1(\tau, \eta)>0$ such that
  for all  $k \ge K_1$, 
  \be
  \label{p_45_13}
  {\frac \epsilon{\lambda}} e^{-\sigma \tau} \int_{-\infty}^\tau \int_{|x| \ge k}
e^{\sigma \xi}   |f(x, \xi) |^2 dx d\xi
+ {\frac \epsilon{\gamma}} e^{-\sigma \tau} \int_{-\infty}^\tau \int_{|x| \ge k}
e^{\sigma \xi}   |g(x, \xi) |^2 dx d\xi
\le \left ({\frac 1\lambda} + {\frac 1\gamma} \right ) \epsilon \eta.
\ee
For the last term on the right-hand side of 
\eqref{p_45_11},  it follows  from  Lemma \ref{lem42} 
that there is
$T_2 =T_2(\tau, D)>0$ such that for all $t \ge T_2$,
$$
{\frac {\epsilon M}{k}} e^{-\sigma \tau} \int_{\tau -t}^\tau \ 
e^{\sigma \xi}  \left (
 \| u (\xi, \tau -t, u_0(\tau -t ) )\|^2 
 + \| \nabla u (\xi, \tau -t, u_0(\tau -t ) )\|^2 
\right )
d \xi
$$
$$
\le {\frac {\epsilon C}{k}} e^{-\sigma \tau}
\int_{-\infty}^\tau  e^{\sigma \xi}
\left (
\| f(\xi) \|^2 + \| g(\xi) \|^2 \right ) d \xi.
$$
Therefore, there is  $K_2 =K_2(\tau, \eta)>0$ such that
for all $k\ge K_2$ and $t \ge T_2$,
\be
\label{p_45_14}
{\frac {\epsilon M}{k}} e^{-\sigma \tau} \int_{\tau -t}^\tau \ 
e^{\sigma \xi}  \left (
 \| u (\xi, \tau -t, u_0(\tau -t ) )\|^2 
 + \| \nabla u (\xi, \tau -t, u_0(\tau -t ) )\|^2 
\right )
d \xi \le \eta.
\ee
Let $K=\max \{K_1, K_2\}$ and $T=\max \{T_1, T_2 \}$.
Then by \eqref{p_45_11}-\eqref{p_45_14} we find that
there exists a positive constant $C_1$ (independent of $\eta$)
such that
for all $k\ge K$ and $t \ge T$,
$$
 \int \theta ({\frac {|x|^2}{k^2}}) \left ( \epsilon |u(x, \tau, \tau -t, u_0(\tau -t) )|^2
  +  |v(x, \tau, \tau -t, v_0(\tau -t) )|^2  \right ) dx
  \le C_1 \eta,
  $$
  and hence for all $k\ge K$ and $t \ge T$,
  $$
 \int _{|x| \ge \sqrt{2} k} \left ( \epsilon |u(x, \tau, \tau -t, u_0(\tau -t) )|^2
  +  |v(x, \tau, \tau -t, v_0(\tau -t) )|^2  \right ) dx
  $$
  $$ \le 
 \int \theta ({\frac {|x|^2}{k^2}}) \left ( \epsilon |u(x, \tau, \tau -t, u_0(\tau -t) )|^2
  +  |v(x, \tau, \tau -t, v_0(\tau -t) )|^2  \right ) dx
  \le C_1 \eta,
  $$
  which completes the proof.
\end{proof}

 \section{Existence of pullback attractors}  
\setcounter{equation}{0}

In this  section,  we prove, by  Proposition \ref{att},
 the existence of a ${\mathcal{D}}_\sigma$-pullback
 global attractor 
for the  non-autonomous   FitzHugh-Nagumo equations on $\R^n$.
To this end, we  
 need to establish the   ${\mathcal{D}}_\sigma$-pullback asymptotic 
compactness of  $\phi$, which is stated as follows.

 \begin{lem}
\label{lem46}
Suppose  \eqref{44}-\eqref{45}  and \eqref{fcond}-\eqref{ginfinity} hold.
Then $\phi$ is $\mathcal{D}_\sigma$-pullback asymptotically compact  in $\ltwo
\times \ltwo$,
that is, for  every $\tau \in \R$,  
 $D=\{D(t)\}_{t\in \R} \in {\mathcal{D}}_\sigma$, and $t_n \to \infty$,
 $(u_{0,n}, v_{0,n} ) \in D(\tau -t_n)$, the sequence
 $\phi(t_n, \tau -t_n,  (u_{0,n}, v_{0,n} ) ) $   has a 
   convergent
subsequence in $\ltwo \times \ltwo$.
\end{lem}

\begin{proof}
Given $s\in \R$,
 $t\ge 0$ and  $(u_0, v_0 ) \in \ltwo \times \ltwo$, define
 $$
 \phi_1 (t, s,  (u_0, v_0) ) = (0, v_1(t+s, s, v_0) )
 \quad \mbox{and} \quad 
 \phi_2 (t,s, (u_0, v_0) ) = (u(t+s, s, u_0), v_2(t+s, s, 0) ),
$$
where $v_1$ and $v_2$ are solutions to \eqref{v1}
and \eqref{v2}, respectively,  and $(u,v)$  with $v=v_1 + v_2$
is the solution of problem \eqref{41}-\eqref{43}.
It is clear that
$$
 \phi  (t, s,  (u_0, v_0) ) =   \phi_1 (t, s, (u_0, v_0) )
 +  \phi_2 (t, s,  (u_0, v_0) ),
 $$ and hence
\be
\label{p46_1}
 \phi  (t_n, \tau -t_n ,  (u_{0,n}, v_{0,n} ) ) =   \phi_1 (t_n,  \tau -t_n ,  (u_{0,n},  v_{0,n}) )
 +  \phi_2 (t_n ,  \tau -t_n ,  (u_{0,n}, v_{0,n}) ) .
\ee
By \eqref{v1_ener} we get that
\be
\label{p46_2}
 \| \phi_1 (t_n,  \tau -t_n ,  (u_{0,n},  v_{0,n}) ) \|
 =e^{-\epsilon \gamma \tau} e^{\epsilon \gamma (\tau - t_n) }
 \| v_0(\tau - t_n ) \|
 \to 0 \quad \mbox{as} \  n \to \infty.
 \ee
Form \eqref{p46_1}-\eqref{p46_2} it follows that
the sequence $\phi  (t_n, \tau -t_n ,  (u_{0,n}, v_{0,n} ) )$ will have
a convergent subsequence in $\ltwo \times \ltwo$ as long as
 $ \phi_2 (t_n ,  \tau -t_n ,  (u_{0,n}, v_{0,n})  )$ is precompact.
 Next we use the uniform estimates on the tails of solutions to establish
 the precompactness of  $\phi_2 (t_n ,  \tau -t_n ,  (u_{0,n}, v_{0,n})  )$
 in $\ltwo \times \ltwo$, that is,  we will prove that for every
 $\eta>0$,  the sequence
 $\phi_2 (t_n ,  \tau -t_n ,  (u_{0,n}, v_{0,n}) ) $ has a finite covering
 of balls of radii less than $\eta$. 
 Given $K>0$,   denote by 
\[{\Omega}_K = \{ x: |x|     \le K \} \quad \mbox{and} \quad
        {\Omega}^c_K = \{ x: |x|     >  K \}.\]
Then by  Lemma  \ref{lem45},    given  $\eta >0$,
there exist $K= K(\tau, \eta)>0$  and $T=T(\tau, D, \eta)>0$
such that for $t \ge T$, 
$$
\|   \phi  (t  ,  \tau -t  ,  (u_{0}(\tau-t) , v_{0} (\tau -t) ) )  \| _{L^2 ({\Omega}^c_{K} )
\times L^2 ( {\Omega}^c_{K} )}
\le \frac{\eta}{8}.
$$
Since $t_n \to \infty$, there is $N=N(\tau, D, \eta)>0$ such that
$t_n \ge T$ for all $n \ge N$, and hence we obtain that, 
for all  $n  \ge N$, 
\be
\label{p46_3}
\|   \phi  (t_n  ,  \tau -t_n  ,  (u_{0,n}  , v_{0,n}   ) )  \| _{L^2 ({\Omega}^c_{K} )
\times L^2 ( {\Omega}^c_{K} )}
\le \frac{\eta}{8}.
\ee
It follows from \eqref{p46_1}-\eqref{p46_3} that there is $N_1 =N_1(\tau, D, \eta)$ such that
for all $n \ge N_1$,
\be
\label{p46_4}
\|   \phi _2  (t_n  ,  \tau -t_n  ,  (u_{0,n}  , v_{0,n}   ) )   \| _{L^2 ({\Omega}^c_{K} )
\times L^2 ( {\Omega}^c_{K} )}
\le \frac{\eta}{4}.
\ee
On the other hand, by Lemmas  \ref{lem44} and \ref{lemv2},    there
exist $C=C(\tau, D)>0$ and $N_2(\tau, D)>0$ such that for all $n \ge N_2 $,
\be
\label{p46_5}
\|   \phi_2  (t_n  ,  \tau -t_n  ,  (u_{0,n}  , v_{0,n}   ) )  \| _{H^1( {\Omega}_{K } )
\times H^1( {\Omega}_{K } )} \le C.
\ee
By the compactness of embedding
$ H^1( {\Omega}_{K } )  \hookrightarrow  L^2 ( {\Omega}_{K } )$,
 the sequence $  \phi _2  (t_n  ,  \tau -t_n  ,  (u_{0,n}  , v_{0,n}   ) ) $ is precompact in
$L^2 ( {\Omega}_{K } )
\times
L^2 ( {\Omega}_{K } )$.
Therefore,
for the given $\eta>0$,
$ \phi _2  (t_n  ,  \tau -t_n  ,  (u_{0,n}  , v_{0,n}   ) )$ has a finite covering in
$L^2 ( {\Omega}_{K } )
 \times L^2 ( {\Omega}_{K })$ of
balls of radii less than $\eta/4$, which along with  \eqref{p46_4}  shows
that $    \phi _2  (t_n  ,  \tau -t_n  , (u_{0,n}  , v_{0,n}   ) )   $ has a finite covering in
$ \ltwo  \times  \ltwo $ of balls of radii less than $\eta$, and
thus  $  \phi _2  (t_n  ,  \tau -t_n  ,  (u_{0,n}  , v_{0,n}   ) )$  is precompact
in   $\ltwo \times  \ltwo $.   The proof is completed.
  \end{proof}

   We are now ready to prove the existence of a pullback  attractor 
   for the $\theta$-cocycle $\phi$.

 \begin{thm}
\label{thm47}
Suppose  \eqref{44}-\eqref{45}  and \eqref{fcond}-\eqref{ginfinity} hold.
Then problem \eqref{41}-\eqref{43}  has  a  unique $\mathcal{D}_\sigma$-pullback  global attractor 
$\{\mathcal{A}(\tau) \}_{\tau \in \R}$  in $\ltwo
\times \ltwo$.
  \end{thm}

\begin{proof}
For $ \tau \in \R$, denote by
$$
B(\tau) = \{ (u,v) \in \ltwo \times \ltwo: \ \| u \| ^2 + \|  v \|^2 
\le M e^{-\sigma \tau} \int_{-\infty}^\tau
e^{\sigma \xi} (\| f (\xi) \|^2 + \| g(\xi) \|^2 ) d \xi \},
$$
where $M$ is the constant in Lemma \ref{lem42}.
Note that $B=\{B(\tau)\}_{\tau \in \R} \in {\mathcal{D}_\sigma} $
is a ${\mathcal{D}_\sigma}$-pullback absorbing  for
$\phi$ in $\ltwo \times \ltwo$ by Lemma \ref{lem42}.
On the other hand, $\phi$ is ${\mathcal{D}_\sigma}$-pullback
asymptotically compact by Lemma \ref{lem46}. Thus the existence of 
a ${\mathcal{D}_\sigma}$-pullback global attractor for $\phi$ follows
from   Proposition \ref{att}  immediately.
  \end{proof}

\section{Uniform  bounds  of  attractors in $\epsilon$}
\setcounter{equation}{0}

  In this  section, we investigate the limiting behavior of the random
 attractor $\{\mathcal{A}(\tau)\}_{\tau \in \R}$  for problem \eqref{41}-\eqref{43}
 when  the small parameter $\epsilon \to 0$. To indicate the fact
 that the random attractor depends on $\epsilon$, hereafter we  write the random
 attractor as $\{\mathcal{A}^\epsilon (\tau)\}_{\tau \in \R}$
 instead of  $\{\mathcal{A}(\tau)\}_{\tau \in \R}$.
 Note that when $\epsilon =0$, \eqref{42} reduces to ${\frac {dv}{dt}} =0$, and hence
 $v$ is conserved in this case. This shows that the limiting system with $\epsilon =0$
 has no global attractor in $\ltwo \times \ltwo$. Based on this fact, one may guess that
 the random attractor $\{\mathcal{A}^\epsilon (\tau)\}_{\tau \in \R}$
 blows up as $\epsilon \to 0$.
 In this respect, we will show that the limiting behavior of  $\{\mathcal{A}^\epsilon (\tau)\}_{\tau \in \R}$
 heavily depends on the behavior of the external terms $f$ and $g$.
 If $f$ or $g$ is unbounded with respect to time in $\ltwo$, then it is very likely
 that $\{\mathcal{A}^\epsilon (\tau)\}_{\tau \in \R}$ blows up.
 However, if both $f$ and $g$ are bounded,  
  $ \mathcal{A}^\epsilon (\tau) $ are uniformly bounded
 in $\ltwo \times \ltwo$ with  respect to $\epsilon$, that is
 $\{\mathcal{A}^\epsilon (\tau)\}_{\tau \in \R}$ does not blow up in this case.

  It follows from \eqref{p42_10} that for every $\tau \in \R$, there is
  $T=T(\tau)>0$ such that for every  $t\ge $T and
$(u_0, v_0) \in {\mathcal{A}}^\epsilon (\tau -t)$,
$$
\| v(\tau, \tau -t, v_0(\tau -t) ) \|^2
\le e^{-\sigma \tau}
\left (
 {\frac {2\epsilon}{\lambda}} \int_{-\infty}^\tau e^{\sigma \xi} \| f(\xi) \|^2 d \xi
 +
 {\frac {2\epsilon}{\gamma}} \int_{-\infty}^\tau e^{\sigma \xi} \| g(\xi) \|^2 d \xi
\right ),
$$
which implies that, for $\tau =0$, $t \ge T$ and $(u_0, v_0) \in {\mathcal{A}}^\epsilon (-t)$,
\be
\label{limit1}
\| v(0,   -t, v_0(-t) ) \|^2
\le  
 {\frac {2\epsilon}{\lambda}} \int_{-\infty}^0 e^{\sigma \xi} \| f(\xi) \|^2 d \xi
 +
 {\frac {2\epsilon}{\gamma}} \int_{-\infty}^0 e^{\sigma \xi} \| g(\xi) \|^2 d \xi
.
\ee
  Next we illustrate that the right-hand side of 
  \eqref{limit1} is unbounded  as $\epsilon \to 0$
  if $f$ or $g$ is unbounded in $\ltwo$. To this end,
  we take
  \be
\label{fg_unbd}
 f(x,t) = \sqrt{|t|}  \ f_1(x) \quad \mbox{and} \quad 
  g(x,t) = \sqrt{|t|} \  g_1(x), \quad x\in \R^n, \ \ t\in \R,
 \ee
  where $f_1$ and $g_1$ are given in $\ltwo$.
  It is clear that $f$ and $g$ are unbounded in $\ltwo$ as
   $t \to \pm\infty$. In this case, the right-hand side of \eqref{limit1}
   is given by
   $$
   {\frac {2\epsilon}{\lambda}} \int_{-\infty}^0 e^{\sigma \xi} \| f(\xi) \|^2 d \xi
 +
 {\frac {2\epsilon}{\gamma}} \int_{-\infty}^0 e^{\sigma \xi} \| g(\xi) \|^2 d \xi
$$
$$
=
- {\frac {2\epsilon}{\lambda}} \| f_1 \|^2  \int_{-\infty}^0 e^{\sigma \xi} \xi  d \xi
-
 {\frac {2\epsilon}{\gamma}} \| g_1 \|^2  \int_{-\infty}^0 e^{\sigma \xi} \xi d \xi
 $$
\be
\label{limit2}
 ={\frac {2\epsilon}{\sigma^2}} \left ( {\frac {\|f_1\|^2}\lambda} + {\frac {\| g_1\|^2}\gamma}
 \right ) 
={\frac {8 }{\gamma^2 \epsilon}} \left ( {\frac {\|f_1\|^2}\lambda} + {\frac {\| g_1\|^2}\gamma}
 \right ).
\ee
By \eqref{limit1}-\eqref{limit2} we get that, for $t \ge T$ and $(u_0, v_0) \in {\mathcal{A}}^\epsilon (-t)$, 
\be
\label{limit3}
\| v(0, -t, v_0(-t)) \|^2
\le
 {\frac {8 }{\gamma^2 \epsilon}} \left ( {\frac {\|f_1\|^2}\lambda} + {\frac {\| g_1\|^2}\gamma}
 \right ).
\ee
 Note that the invariance of  
  $\{\mathcal{A}^\epsilon (\tau)\}_{\tau \in \R}$
  implies that
  \be
  \label{limit4}
  \phi (t, -t,  {\mathcal{A}}^\epsilon (-t) ) = \mathcal{A}^\epsilon (0), \quad \forall \ t \ge 0.
  \ee
  Let $(\tilde{u}, \tilde{v} )$ be an arbitrary element in
  $\mathcal{A}^\epsilon (0)$  and  $t_n \to \infty$. Then it follows
  from \eqref{limit4} that for every $n \ge 1$, there exists
  $(u_{0,n}, v_{0,n}) \in \mathcal{A}^\epsilon (-t_n)$ such that
  $$
  \phi(t_n, -t_n,  (u_{0,n}, v_{0,n} ) ) = (\tilde{u}, \tilde{v})
 ,
 $$
 which implies that
 \be
 \label{limit5}
 ( u(0, -t_n, u_{0,n} ), v(0, -t_n, v_{0,n} ) )
 = (\tilde{u}, \tilde{v}), \quad \forall \ n \ge 1.
 \ee
 Since $t_n \to \infty$, there is $N>0$ such that
 $t_n \ge T$ for all $n \ge N$, and hence by \eqref{limit3}
 and \eqref{limit5} we have, for all $n \ge N$,
\be
\label{limit6}
 \| \tilde{v} \|^2 = \| v(0, -t_n, v_{0,n} ) \|^2
 \le 
 {\frac {8 }{\gamma^2 \epsilon}} \left ( {\frac {\|f_1\|^2}\lambda} + {\frac {\| g_1\|^2}\gamma}
 \right ) .
\ee
Since  the right-hand side of \eqref{limit6} approaches infinity
as $\epsilon \to 0$ and 
 $(\tilde{u}, \tilde{v} )$ is an arbitrary point in 
$\mathcal{A}^\epsilon (0)$,  we find that
the upper bound for  $v$ components
of  $\mathcal{A}^\epsilon (0)$  becomes
unbounded  as $\epsilon \to 0$. This shows that
it is very likely  that $\mathcal{A}^\epsilon (0)$ 
blows up as $\epsilon \to 0$ for such unbounded
$f$ and $g$ given in \eqref{fg_unbd}.
Now the question is what happens if $f$ and $g$ are
bounded in $\ltwo$.  As we will see later, in this case, we can show that
the right-hand side of \eqref{limit1} is uniformly
bounded in $\epsilon$, and hence the random attractor does not blows  up.
To prove this result,  we   need the uniform estimates of solutions
with respect to $\epsilon$.

In this  section, we  agree  that   $K_i$  $ (i \in \N)$
are any positive constants  which depend  only on the data     $(
\nu, \lambda,  \gamma)$, but  not  on
$\epsilon$; while   $C_i$   $ (i \in \N)$  are any
positive constants which may depend on  the parameters
$\epsilon$,  $ \nu, \lambda $ and  $  \gamma$.

\begin{lem}
\label{lem51}
 Suppose   $f \in L^\infty (\R, \ltwo)$, $g \in L^\infty (\R, H^1(\R^n) )$ and
\eqref{44}-\eqref{45}    hold.
Then for every   $D=\{D(t)\}_{t\in \R} \in {\mathcal{D}}_\sigma$, $\tau \in \R$
and $t\ge 0$, the following holds for all $\xi  \ge \tau -t$,
$$ 
 \| v(\xi, \tau -t, v_0(\tau -t) ) \| ^2 \le    e^{- \sigma  \xi}
 e^{\sigma (\tau -t)}
 \left ( \|u_0(\tau -t) \|^2 + \| v_0 (\tau -t ) \|^2
 \right ) + K,
$$
and
$$
\epsilon   \int_{\tau -t}^\tau
e^{\sigma \xi} \| \nabla u (\xi, \tau -t, u_0(\tau -t) ) \|^2 d\xi
\le
{\frac 1{2\nu}}  e^{\sigma (\tau -t)} \left (
\| u_0(\tau -t) \|^2 + \| v_0(\tau -t) \|^2 \right )
+K e^{\sigma \tau},
$$
where $K$ is a  positive constant depending   on the data     $(
\nu,  \lambda,   \gamma)$,  but not on $\epsilon$ or $\tau$.
\end{lem}

\begin{proof}
 Since $f$ and $g$ are bounded in $\ltwo$ and $H^1(\R^n)$, respectively,
 by \eqref{p42_8} we have
 $$
 {\frac d{dt}} \left (
  \epsilon \| u \|^2 + \| v \|^2
 \right )
 + \sigma \left (
  \epsilon \| u \|^2 + \|  v \|^2
 \right )
 + 2 \epsilon \nu \| \nabla u \|^2 
 \le K_1 \epsilon.
 $$
 Multiplying the above by $e^{\sigma t}$ and then integrating  the resulting inequality
over
 $(\tau -t, \xi)$, we get 
 $$
 e^{\sigma \xi}  \left (
  \epsilon \| u(\xi, \tau -t, u_0(\tau -t) ) \|^2 + \| v(\xi, \tau -t, v_0(\tau -t )) \|^2
 \right )
 $$
 $$
 + 2 \epsilon \nu \int_{\tau -t}^\xi e^{\sigma s}
 \| \nabla u(s, \tau -t, u_0(\tau -t) ) \|^2 ds
 $$
 $$
 \le e^{\sigma (\tau -t) }
 \left (
  \epsilon \| u_0(\tau -t) \|^2 + \|  v_0 (\tau -t )\|^2
 \right )
 + K_1\epsilon \int_{\tau-t}^\xi e^{\sigma s} d s
 $$
 $$
 \le e^{\sigma (\tau -t) }
 \left (
   \| u_0(\tau -t) \|^2 + \|  v_0 (\tau -t )\|^2
 \right )
 + K_1\epsilon \int_{-\infty}^{\xi} e^{\sigma s} d s
 $$
\be
\label{p51_1}
 \le e^{\sigma (\tau -t) }
 \left (
   \| u_0(\tau -t) \|^2 + \|  v_0 (\tau -t )\|^2
 \right )
 + {\frac {K_1\epsilon }{\sigma}}e^{\sigma  \xi} .
\ee
Note that $\sigma ={\frac 12} \epsilon \gamma$. 
 Then it follows from \eqref{p51_1}
that
$$  \| v(\xi, \tau -t, v_0(\tau -t )) \|^2
+ 2 \epsilon \nu e^{-\sigma \xi}  \int_{\tau -t}^\xi e^{\sigma s}
 \| \nabla u(s, \tau -t, u_0(\tau -t) ) \|^2 ds
 $$
\be
\label{p51_2}
 \le  e^{-\sigma \xi} e^{\sigma (\tau -t) }
 \left (
   \| u_0(\tau -t) \|^2 + \|  v_0 (\tau -t )\|^2
 \right )
 +
{\frac {2K_1}{\gamma}}.
\ee
Particularly, if $\xi =\tau$, by \eqref{p51_2} we get that
$$
\epsilon e^{-\sigma \tau} \int_{\tau -t}^\tau
e^{\sigma \xi} \| \nabla u (\xi, \tau -t, u_0(\tau -t) ) \|^2 d\xi
$$
$$
\le
{\frac 1{2\nu}} e^{-\sigma \tau} e^{\sigma (\tau -t)} \left (
\| u_0(\tau -t) \|^2 + \| v_0(\tau -t) \|^2 \right )
+{\frac {K_1}{\nu \gamma}},
$$
 which along with \eqref{p51_2}  completes the proof.
\end{proof}

\begin{lem}
\label{lem52} 
Suppose   $f \in L^\infty (\R, \ltwo)$, $g \in L^\infty (\R, H^1(\R^n) )$ and
\eqref{44}-\eqref{45}    hold.
Then for every   $D=\{D(t)\}_{t\in \R} \in {\mathcal{D}}_\sigma$, $\tau \in \R$
and $t\ge 0$,  we have
$$ 
 \|  u (\tau, \tau -t, u_0(\tau -t) ) \| ^2 
 + 2 \nu e^{-\lambda \tau}
 \int_{\tau -t}^\tau e^{\lambda \xi} \| \nabla u(\xi, \tau -t, u_0(\tau -t) ) \|^2 d \xi
 $$
 $$
\le C   e^{- \sigma  \tau}
 e^{\sigma (\tau -t)}
 \left ( \|u_0(\tau -t) \|^2 + \| v_0 (\tau -t ) \|^2
 \right ) + K,
$$
where $K$ is a  positive constant depending   on the data     $(
\nu,  \lambda,   \gamma)$,  but not on $\epsilon$ or $\tau$;
while $C$    depends   on the data     $(
\nu,  \lambda,   \gamma)$ as well as  $\epsilon$, but not on $\tau$.
\end{lem}

\begin{proof}
 Taking the inner product of \eqref{41} with $u$ in $\ltwo$, we
 get that
 \be
 \label{p52_1}
 {\frac 12} {\frac d{dt}} \| u \|^2
 + \nu \| \nabla u \|^2 + \lambda \| u \|^2
 + (h(u), u)  = - (u,v) + (f(t), u).
 \ee
 Note that   the right-hand side of \eqref{p52_1} is bounded  by
 \be
 \label{p52_2}
 \|u\| \| v\| + \|f(t)\| \| u \|
   \le {\frac 12} \lambda \| u \|^2 + {\frac 1\lambda} \| v \|^2
 + {\frac 1\lambda} \| f(t) \|^2 .
 \ee
 By \eqref{p52_1}-\eqref{p52_2} and \eqref{44}, we obtain that,
 \be
 \label{p52_3}
 {\frac d{dt}} \| u \|^2 + 2 \nu \| \nabla u \|^2 + \lambda \| u
 \|^2
 \le {\frac 2\lambda} \| v \|^2 +
 {\frac 2\lambda} \| f(t) \|^2 .
 \ee
 Multiplying \eqref{p52_3} by $e^{\lambda t}$ and then integrating
 the resulting inequality over $(\tau -t, \tau)$ with $t \ge 0$, we obtain
 that
 $$
 \| u(\tau, \tau-t, u_0(\tau -t) ) \|^2
 + 2 \nu  e^{-\lambda \tau} \int_{\tau -t}^\tau e^{\lambda \xi} \| \nabla u(\xi,  \tau -t, u_0(\tau -t) ) \|^2 d \xi
 $$
 \be
 \label{p52_4}
 \le
 e^{-\lambda t} \| u_0 (\tau -t ) \|^2
 +{\frac 2\lambda}e^{-\lambda \tau}
 \int_{\tau -t}^\tau
 e^{\lambda \xi} \| v(\xi, \tau -t, v_0(\tau -t) )\|^2 d\xi
 +{\frac 2\lambda}e^{-\lambda \tau}
 \int_{\tau -t}^\tau
 e^{\lambda \xi} \|  f(\xi ) \|^2 d\xi .
\ee
 Note that $f \in L^\infty(\R, \ltwo)$.  By  \eqref{p52_4}   and
Lemma \ref{lem51} we find
 that
 $$
 \| u(\tau, \tau-t, u_0(\tau -t) ) \|^2
 + 2 \nu  e^{-\lambda \tau} \int_{\tau -t}^\tau e^{\lambda \xi} \| \nabla u(\xi,  \tau -t, u_0(\tau -t) ) \|^2 d \xi
 $$
 $$
 \le 
   e^{-\lambda t} \| u_0 (\tau -t ) \|^2
+ K_1 e^{-\lambda \tau}  \int_{\tau -t}^\tau e^{\lambda \xi} d\xi
$$
$$
 + K_2 e^{-\lambda \tau}  e^{\sigma (\tau -t) }
 \left (
  \| u_0(\tau -t )\|^2 + \| v_0(\tau -t ) \|^2
 \right ) 
 \int_{\tau -t}^\tau e^{(\lambda -\sigma) \xi} d\xi.
 $$
 $$
 \le
  e^{-\lambda t} \| u_0 (\tau -t ) \|^2
  +K_3
  + {\frac {K_2}{\lambda -\sigma}} e^{-\sigma t}
  \left (
  \| u_0(\tau -t )\|^2 + \| v_0(\tau -t ) \|^2
 \right ) .
  $$
  Note that $\lambda >\sigma$.  Then  it  follows 
  from the above that
  $$
 \| u(\tau, \tau-t, u_0(\tau -t) ) \|^2
 + 2 \nu  e^{-\lambda \tau} \int_{\tau -t}^\tau e^{\lambda \xi} \| \nabla u(\xi,  \tau -t, u_0(\tau -t) ) \|^2 d \xi
 $$
 \be
 \label{p52_10}
 \le
 K_3 + (1+ {\frac {K_2} {\lambda -\sigma}} )  e^{-\sigma t}
  \left (
  \| u_0(\tau -t )\|^2 + \| v_0(\tau -t ) \|^2
 \right ) ,
\ee
 which completes the proof.
\end{proof}

\begin{lem}
\label{lem53} 
Suppose   $f \in L^\infty (\R, \ltwo)$, $g \in L^\infty (\R, H^1(\R^n) )$ and
\eqref{44}-\eqref{45}    hold.
Then for every   $D=\{D(t)\}_{t\in \R} \in {\mathcal{D}}_\sigma$, $\tau \in \R$
and $t\ge  1$,  we have
$$ 
  e^{-\lambda \tau}
 \int_{\tau -1}^\tau e^{\lambda \xi} \| \nabla u(\xi, \tau -t, u_0(\tau -t) )  \|^2 d \xi
\le C   e^{- \sigma  \tau}
 e^{\sigma (\tau -t)}
 \left ( \|u_0(\tau -t) \|^2 + \| v_0 (\tau -t ) \|^2
 \right ) + K,
$$
where $K$ is a  positive constant depending   on the data     $(
\nu,  \lambda,   \gamma)$,  but not on $\epsilon$ or $\tau$;
while $C$    depends   on the data     $(
\nu,  \lambda,   \gamma)$ as well as  $\epsilon$, but not on $\tau$.
\end{lem}

\begin{proof}
 By \eqref{p52_3} we find that
\be\label{p53_1}
 {\frac d{dt}} \| u \|^2 +\lambda \| u \|^2
 \le {\frac 2\lambda} \| v \|^2 + {\frac 2\lambda} \| f(t) \|^2.
 \ee
 Using $f\in L^\infty(\R, \ltwo )$ and repeating the proof of
 \eqref{p52_10} we  can get  from \eqref{p53_1} that
 \be
 \label{p53_2}
 \| u(\tau-1, \tau -t, u_0(\tau -t) ) \|^2
 \le C_1  e^{-\sigma \tau} e^{\sigma (\tau -t ) }
 \left (
  \| u_0(\tau-t) \|^2 + \| v_0(\tau -t ) \|^2\right )
  +K_1.
 \ee
 Integrating  \eqref{p52_3} over $(\tau -1, \tau)$, by Lemma \ref{lem51}
 we have
 $$
 e^{\lambda \tau}  \| u(\tau, \tau -t, u_0(\tau -t) ) \|^2
 + 2\nu \int^\tau_{\tau -1} e^{\lambda \xi} \| \nabla u(\xi, \tau -t, u_0(\tau -t ) ) \|^2 d\xi
 $$
 $$
 \le e^{\lambda (\tau -1)} \| u(\tau-1, \tau -t, u_0(\tau -t) ) \|^2
 $$
 $$
 + {\frac 2\lambda}\int_{\tau -1}^\tau e^{\lambda \xi} \| v(\xi, \tau -t, v_0(\tau -t))\|^2 d\xi
 +  \int_{\tau -1}^\tau e^{\lambda \xi} \| f(\xi) \|^2  d\xi
 $$
 $$
 \le e^{\lambda (\tau -1)} \| u(\tau-1, \tau -t, u_0(\tau -t) ) \|^2
 $$
 $$
 + {\frac 2\lambda} e^{\sigma (\tau -t)}
 \left (
  \| u_0(\tau -t ) \|^2 + \| v_0 (\tau -t ) \|^2
 \right ) \int^\tau_{\tau -1} e^{(\lambda -\sigma) \xi} d \xi
 + K_2 \int^\tau_{\tau -1} e^{\lambda \xi} d\xi
 $$
$$
 \le e^{\lambda (\tau -1)} \| u(\tau-1, \tau -t, u_0(\tau -t) ) \|^2
 $$
 $$
 + {\frac 2{\lambda (\lambda-\sigma)}} e^{\lambda\tau  -\sigma t}
 \left (
  \| u_0(\tau -t ) \|^2 + \| v_0 (\tau -t ) \|^2
 \right )  
 + {\frac {K_2}\lambda}  e^{\lambda \tau},
 $$
 which along with  \eqref{p53_2} implies that
 $$
 e^{\lambda \tau}  \| u(\tau, \tau -t, u_0(\tau -t) ) \|^2
 + 2\nu \int^\tau_{\tau -1} e^{\lambda \xi} \| \nabla u(\xi, \tau -t, u_0(\tau -t ) ) \|^2 d\xi
 $$
 $$
 \le  C_2  e^{\lambda\tau  -\sigma t}
 \left (
  \| u_0(\tau -t ) \|^2 + \| v_0 (\tau -t ) \|^2
 \right )  
 +  K_3 e^{\lambda \tau}.
 $$
 Then Lemma \ref{lem53} follows from 
 the above immediately.
\end{proof}

Next,  we derive uniform estimates
 in $\epsilon$   for  the   $u$ components of the solutions of
 problem \eqref{41}-\eqref{43} in
   $H^1(\R^n) $.

\begin{lem}
\label{lem54} 
Suppose   $f \in L^\infty (\R, \ltwo)$, $g \in L^\infty (\R, H^1(\R^n) )$ and
\eqref{44}-\eqref{45}    hold.
Then for every   $D=\{D(t)\}_{t\in \R} \in {\mathcal{D}}_\sigma$, $\tau \in \R$
and $t\ge  1$,  we have
$$ 
\| \nabla u(\tau,  \tau -t, u_0(\tau -t) )  \|^2  
\le C   e^{- \sigma  \tau}
 e^{\sigma (\tau -t)}
 \left ( \|u_0(\tau -t) \|^2 + \| v_0 (\tau -t ) \|^2
 \right ) + K,
$$
where $K$ is a  positive constant depending   on the data     $(
\nu,  \lambda,   \gamma)$,  but not on $\epsilon$ or $\tau$;
while $C$    depends   on the data     $(
\nu,  \lambda,   \gamma)$  and   $\epsilon$, but not on $\tau$.
\end{lem}

\begin{proof}
Note that \eqref{p44_5} implies that
$$
{\frac d{dt}} \| \nabla u \|^2 \le K_1 \| \nabla u \|^2
+ {\frac 2\nu} \| v \|^2 + {\frac 2\nu} \| f(t) \|^2.
$$
Since $f \in L^\infty(\R, \ltwo)$ we get that
 \be
 \label{p54_1}
  {\frac d{dt}} \| \nabla u \|^2 + \lambda \| \nabla u \|^2
 \le (\lambda +  K_1 )  \| \nabla u \|^2
+ {\frac 2\nu} \| v \|^2 + K_2.
\ee
Multiplying \eqref{p54_1} by $e^{\lambda t}$ and then integrating
the resulting inequality over $(s, \tau)$ with
$s \in (\tau -1, \tau)$, we find that for all $t \ge 1$,
$$
e^{\lambda \tau} \| \nabla u(\tau, \tau -t, u_0(\tau -t) ) \|^2
\le e^{\lambda s} \| \nabla u(s, \tau -t, u_0(\tau -t) ) \|^2
$$
$$
+ (\lambda +K_1) \int_s^\tau e^{\lambda \xi} \| \nabla u(\xi, \tau -t, u_0(\tau -t) ) \|^2 d\xi
$$
$$
+ {\frac 2\nu} \int_s^\tau e^{\lambda \xi} \| v(\xi, \tau -t, v_0(\tau -t) ) \|^2 d\xi
+K_2\int_s^\tau e^{\lambda \xi}  d\xi
$$
 $$
 \le e^{\lambda s} \| \nabla u(s, \tau -t, u_0(\tau -t) ) \|^2
+ (\lambda +K_1) \int_{\tau -1}^\tau e^{\lambda \xi} \| \nabla u(\xi, \tau -t, u_0(\tau -t) ) \|^2 d\xi
$$
\be
\label{p54_2}
+ {\frac 2\nu} \int_{\tau -1}^\tau e^{\lambda \xi} \| v(\xi, \tau -t, v_0(\tau -t) ) \|^2 d\xi
+ {\frac {K_2}\lambda}  e^{\lambda \tau}.
\ee
We now integrate \eqref{p54_2} with respect to $s$ on $(\tau -1, \tau)$ to get
$$
e^{\lambda \tau} \| \nabla u(\tau, \tau -t, u_0(\tau -t) ) \|^2
\le  \int_{\tau -1}^\tau e^{\lambda s} \| \nabla u(s, \tau -t, u_0(\tau -t) ) \|^2 ds
$$
$$
 + (\lambda +K_1) \int_{\tau -1}^\tau e^{\lambda \xi} \| \nabla u(\xi, \tau -t, u_0(\tau -t) ) \|^2 d\xi
$$
\be
\label{p54_3}
+ {\frac 2\nu} \int_{\tau -1}^\tau e^{\lambda \xi} \| v(\xi, \tau -t, v_0(\tau -t) ) \|^2 d\xi
+ {\frac {K_2}\lambda} e^{\lambda \tau}.
\ee
By Lemma \ref{lem53}, the first two terms on the right-hand side of
\eqref{p54_3} satisfy
$$
\int_{\tau -1}^\tau e^{\lambda s} \| \nabla u(s, \tau -t, u_0(\tau -t) ) \|^2 ds
 + (\lambda +K_1) \int_{\tau -1}^\tau e^{\lambda \xi} \| \nabla u(\xi, \tau -t, u_0(\tau -t) ) \|^2 d\xi
$$
\be
\label{p54_4}
\le C_1 e^{\lambda \tau -\sigma  t} 
\left (
 \| u_0(\tau -t) \|^2 + \|  v_0(\tau -t ) \|^2
\right )
+ K_3 e^{\lambda \tau}.
\ee
On the other hand,  by Lemma \ref{lem51}, for 
the third term on the right-hand side of
\eqref{p54_3} we have
$$
{\frac 2\nu} \int_{\tau -1}^\tau e^{\lambda \xi} \| v(\xi, \tau -t, v_0(\tau -t) ) \|^2 d\xi
 $$
 $$
 \le 
 {\frac 2\nu}  e^{\sigma (\tau -t)} 
 \left (
 \| u_0(\tau -t) \|^2 + \|  v_0(\tau -t ) \|^2
\right ) \int_{\tau -1} ^\tau e^{\lambda -\sigma) \xi} d\xi
+ {\frac 2\nu} K_4 \int_{\tau -1}^\tau e^{\lambda \xi} d\xi
$$
\be \label{p54_5}
 \le 
 {\frac 2{\nu (\lambda -\sigma)}}  e^{\lambda \tau - \sigma  t} 
 \left (
 \| u_0(\tau -t) \|^2 + \|  v_0(\tau -t ) \|^2
\right )  
+ {\frac {2 K_4}{\nu \lambda}} e^{\lambda \tau}  .
\ee
Then it follows from \eqref{p54_3}-\eqref{p54_5} that 
 $$
e^{\lambda \tau} \| \nabla u(\tau, \tau -t, u_0(\tau -t) ) \|^2
\le   C_2  e^{\lambda \tau - \sigma  t} 
 \left (
 \| u_0(\tau -t) \|^2 + \|  v_0(\tau -t ) \|^2
\right )  
+ K_5  e^{\lambda \tau}  ,
$$
which completes the proof.
 \end{proof}

The following result is concerned with the uniform estimates
in $\epsilon$ for    solutions of problem \eqref{v2}.

\begin{lem}
\label{lem55} 
Suppose   $f \in L^\infty (\R, \ltwo)$, $g \in L^\infty (\R, H^1(\R^n) )$ and
\eqref{44}-\eqref{45}    hold.
Then for every   $D=\{D(t)\}_{t\in \R} \in {\mathcal{D}}_\sigma$, $\tau \in \R$
and $t\ge  1$,  we have
$$ 
\| \nabla v_2(\tau,  \tau -t, 0 )  \|^2  
\le C   e^{- \sigma  \tau}
 e^{\sigma (\tau -t)}
 \left ( \|u_0(\tau -t) \|^2 + \| v_0 (\tau -t ) \|^2
 \right ) + K,
$$
where $K$ is a  positive constant depending   on the data     $(
\nu,  \lambda,   \gamma)$,  but not on $\epsilon$ or $\tau$;
while $C$    depends   on the data     $(
\nu,  \lambda,   \gamma)$  and   $\epsilon$, but not on $\tau$.
\end{lem}

\begin{proof}
By $g \in L^\infty(\R, H^1(\R^n))$ and  \eqref{pv2_2} we  get that 
\be
\label{p55_1}
{\frac{d}{dt}} \| \nabla v_2 \| ^2 +     \sigma  \| \nabla v_2 \|^2
\le {\frac {4 \epsilon}\gamma} \| \nabla u \|^2
+
\epsilon K_1.
\ee
Multiplying \eqref{p55_1} by $e^{  \sigma t}$ and then integrating 
the resulting inequality over $(\tau - t, \tau)$,  we  obtain  that
 $$
e^{\sigma \tau}  \| \nabla v_2 ( \tau, \tau -t, 0 ) \|^2
\le
{\frac {4 \epsilon}\gamma}  \int_{\tau-t}^{\tau } e^{\sigma \xi} \| \nabla u(\xi, \tau -t, u_0(\tau -t)) \|^2 d\xi
+  \epsilon K_1    \int^\tau_{\tau -t} e^{\sigma \xi} d\xi .
$$
Note that $\sigma ={\frac 12} \epsilon \gamma$. Then by Lemma \ref{lem51} we find that
$$
e^{\sigma \tau}  \| \nabla v_2 ( \tau, \tau -t, 0 ) \|^2
\le C_1    e^{\sigma (\tau -t)}  \left (
\|  u_0 ( \tau -t) \|^2  + \| v_0(\tau -t ) \|^2 \right )
+ K_2      e^{\sigma \tau} ,
$$
which completes the proof.
\end{proof}

As an immediate consequence of  \eqref{v1_ener} and
Lemmas \ref{lem51}, \ref{lem52}, \ref{lem54} and \ref{lem55},
we have the following  uniform  estimates.

\begin{cor}
\label{cor56}
Suppose   $f \in L^\infty (\R, \ltwo)$, $g \in L^\infty (\R, H^1(\R^n) )$ and
\eqref{44}-\eqref{45}    hold.
Then for every   $D=\{D(t)\}_{t\in \R} \in {\mathcal{D}}_\sigma$, $\tau \in \R$
and $t\ge  1$,  we have
$$ 
\|  u(\tau,  \tau -t, u_0(\tau -t) )  \|^2_{H^1}   + \|  v_2(\tau, \tau -t, 0 ) \|^2_{H^1}
\le C   e^{- \sigma  \tau}
 e^{\sigma (\tau -t)}
 \left ( \|u_0(\tau -t) \|^2 + \| v_0 (\tau -t ) \|^2
 \right ) + K,
$$
where $K$ is a  positive constant depending   on the data     $(
\nu,  \lambda,   \gamma)$,  but not on $\epsilon$ or $\tau$;
while $C$    depends   on the data     $(
\nu,  \lambda,   \gamma)$  and   $\epsilon$, but not on $\tau$.
\end{cor}

We are now ready  to show that
the union of  the random  attractor 
$\{\mathcal{A}^\epsilon  (\tau) \}_{\tau \in \R}$  is bounded in
$ H^1(\R^n) \times H^1(\R^n)$.

\begin{thm}
\label{thm57} 
 Suppose   $f \in L^\infty (\R, \ltwo)$, $g \in L^\infty (\R, H^1(\R^n) )$ and
\eqref{44}-\eqref{45}    hold.
 Let $\epsilon_0 < \min\{1, \frac{\lambda}{ \gamma}\}$
be a fixed positive number. Then
the set $ \bigcup\limits_{0< \epsilon \le \epsilon_0} 
 \bigcup\limits_{\tau \in \R} 
{\mathcal{A}}^\epsilon(\tau)$
is bounded in $ H^1(\R^n) \times H^1(\R^n)$.
More precisely, there exists a constant $K$,
  depending   only on the data     $(
\nu, \lambda,  \gamma)$  but  not on $\epsilon$,
such that for every $\tau \in \R$, 
 $\epsilon$ with $0< \epsilon \le \epsilon_0$
and $(u^{\epsilon, \tau} , v^{\epsilon, \tau} ) \in \mathcal{A}^\epsilon (\tau)$,
$$ \| u^{\epsilon , \tau}  \|_{H^1(\R^n)} + \| v^{\epsilon , \tau} \|_{H^1(\R^N)}  \le K.
$$
\end{thm}

\begin{proof}
For fixed  $\epsilon >0$, denote the solutions of 
\eqref{v1} and \eqref{v2} by
$v_1^\epsilon$ and $v_2^\epsilon$, respectively. 
Then  for every $\tau \in \R$ and $t\ge 0$,
the solution  $(u^\epsilon, v^\epsilon)$
of problem \eqref{41}-\eqref{43} with
initial condition $(u_0, v_0)$ at $\tau -t$ 
can be written as, for all $\xi \ge \tau -t$, 
\be
\label{p57_1}
(u^\epsilon (\xi, \tau -t, u_0), v^\epsilon (\xi, \tau-t, v_0))  =
(u^\epsilon (\xi, \tau -t, u_0),  v_2^\epsilon (\xi, \tau -t, 0) ) + (0, v_1^\epsilon (\xi, \tau -t, v_0 )).
\ee
Take a sequence $\{ t_n \}_{n=1}^\infty$ such that
$t_n \ge 1$ and $t_n \to \infty$. Then given $\tau \in \R$
and $(u^{\epsilon , \tau}, v^{\epsilon , \tau} ) \in {\mathcal{A}}^\epsilon (\tau)$,
by the invariance of the random attractor, we find that
there exists a sequence $\{ (u_0^\epsilon (\tau -t_n), v_0^\epsilon (\tau -t_n) )\}
\in {\mathcal{A}}^\epsilon (\tau -t_n)$ such that
\be
\label{p57_2}
(u^{\epsilon , \tau },  v^{\epsilon , \tau })
= (u^\epsilon (\tau, \tau -t_n, u_0^\epsilon  (\tau -t_n) ), 
v^\epsilon(\tau, \tau-t_n, v_0^\epsilon (\tau -t_n) ) ).
\ee
It follows from \eqref{p57_1}-\eqref{p57_2} that
\be
\label{p57_22}
(u^{\epsilon , \tau },  v^{\epsilon , \tau })
=
 (u^\epsilon (\tau, \tau -t_n, u_0^\epsilon  (\tau -t_n) ), 
v_2^\epsilon(\tau, \tau-t_n,  0 ) )
+ (0, v_1^\epsilon (\tau, \tau -t_n, v_0^\epsilon (\tau -t_n) ) ).
\ee
By Corollary \ref{cor56} we have  for all $n \ge 1$,
$$
\|  u^\epsilon (\tau,  \tau -t_n, u_0^\epsilon (\tau -t_n) )  \|^2_{H^1}   
+ \|  v_2^\epsilon (\tau, \tau -t_n, 0 ) \|^2_{H^1}
$$
\be
\label{p57_3}
\le C   e^{- \sigma  \tau}
 e^{\sigma (\tau -t_n)}
 \left ( \|u_0(\tau -t_n) \|^2 + \| v_0 (\tau -t_n ) \|^2
 \right ) + K,
\ee
where $K$ is a  positive constant depending   on the data     $(
\nu,  \lambda,   \gamma)$,  but not on $\epsilon$ or $\tau$.
Note that the first term on the right-hand side of
\eqref{p57_3} approaches zero as $n \to \infty$, and hence
there exists $N=N(\epsilon, \tau)$ such that for all $n \ge N$,
\be
\label{p57_4}
\|  u^\epsilon (\tau,  \tau -t_n, u_0^\epsilon (\tau -t_n) )  \|^2_{H^1}   
+ \|  v_2^\epsilon (\tau, \tau -t_n, 0 ) \|^2_{H^1}
\le 2 K,
\ee
which implies that there is
$ ( \tilde{u}^{\epsilon, \tau}, \tilde{v}^{\epsilon, \tau} )
\in H^1(\R^n) \times H^1(\R^n)$ such that, up to
a subsequence,
\be
\label{p57_5}
 ( u^\epsilon (\tau,  \tau -t_n, u_0^\epsilon (\tau -t_n) ) , 
  v_2^\epsilon (\tau, \tau -t_n, 0 )   ) \to 
( \tilde{u}^{\epsilon, \tau}, \tilde{v}^{\epsilon, \tau} )
\  \mbox{ weakly in  } H^1(\R^n) \times H^1(\R^n),
\ee
as $n \to \infty$. By \eqref{p57_4}-\eqref{p57_5} we have
\be
\label{p57_6}
\| ( \tilde{u}^{\epsilon, \tau}, \tilde{v}^{\epsilon, \tau} )\|_{H^1 \times H^1}
\le
\liminf_{n\to \infty}
\|  ( u^\epsilon (\tau,  \tau -t_n, u_0^\epsilon (\tau -t_n) ) , 
  v_2^\epsilon (\tau, \tau -t_n, 0 )   )\|_{H^1\times H^1}
  \le \sqrt{2K}.
  \ee
  Since 
 the weak convergence in $ H^1(\R^n) \times H^1(\R^n)$ implies the weak convergence
 in $\ltwo \times \ltwo$, by \eqref{p57_5} we have
 \be
 \label{p57_7}
 ( u^\epsilon (\tau,  \tau -t_n, u_0^\epsilon (\tau -t_n) ) , 
  v_2^\epsilon (\tau, \tau -t_n, 0 )   ) \to 
( \tilde{u}^{\epsilon, \tau}, \tilde{v}^{\epsilon, \tau} )
\  \mbox{ weakly in  } \ltwo \times \ltwo.
\ee
On the other hand, by \eqref{v1_ener} we find that
\be
\label{p57_8}
\| v_1^\epsilon (\tau , \tau - t_n, v_0^\epsilon (\tau -t_n) ) \|
= e^{-\epsilon \gamma \tau}
e^{\epsilon \gamma (\tau - t_n) } \| v_0^\epsilon (\tau -t_n) \|
\to 0.
\ee
Taking the weak limit  of \eqref{p57_22} in $\ltwo \times \ltwo$ as $n \to \infty$,
 by \eqref{p57_7} and \eqref{p57_8} we obtain
 \be
 \label{p57_9}
 (u^{\epsilon, \tau}, v^{\epsilon, \tau} ) =  ( \tilde{u}^{\epsilon, \tau},  \tilde{v}^{\epsilon, \tau} ).
 \ee
 Then it follows from \eqref{p57_6} and \eqref{p57_9} that, for  every $\epsilon>0$,
 $\tau \in \R$ and $(u^{\epsilon, \tau}, v^{\epsilon, \tau} ) \in {\mathcal{A}}^\epsilon (\tau)$,
 $$
 \|  (u^{\epsilon, \tau}, v^{\epsilon, \tau} ) \|_{H^1 \times H^1} \le \sqrt{2K}.
$$
 Note that $K $ is independent of $\epsilon$ and $\tau$, and thus the proof is completed.
  \end{proof}

 \end{document}